%% file: main.tex
\documentclass[12pt]{article}
\usepackage{amsmath}
\usepackage{graphicx}
\usepackage{enumerate}
\usepackage{enumitem}
\usepackage{url} 

\usepackage[scriptsize]{subfigure}
\usepackage[font=footnotesize,labelfont=bf]{caption}

\usepackage{amsthm,amssymb}

\RequirePackage[OT1]{fontenc}

\usepackage[hyperindex,breaklinks]{hyperref}
\usepackage[authoryear]{natbib}

\usepackage{bm}
\usepackage{amsfonts}
\usepackage{dsfont}

\numberwithin{equation}{section}
\theoremstyle{plain}
\newtheorem{thm}{Theorem}[section]
\newtheorem*{thm*}{Theorem}
\newtheorem{lemma}[thm]{Lemma}
\newtheorem{proposition}[thm]{Proposition}

\theoremstyle{definition}

\newtheorem{definition}[thm]{Definition}
\newtheorem*{definition*}{Definition}
\newtheorem{assumption}[thm]{Assumption}
\newtheorem{remark}[thm]{Remark}
\newtheorem*{remark*}{Remark}

\newcommand{\F}{\mathcal{F}}

\newcommand{\E}{\mathbb{E}}
\newcommand{\PP}{\mathbb{P}}

\newcommand{\II}{\mathds{1}}

\newcommand{\ul}{\underline}

\newcommand{\field}[1]{\mathds{#1}}
\newcommand{\R}{\field{R}}
\newcommand{\C}{\field{C}}

\newcommand{\N}{\field{N}}

\newcommand{\Var}{\text{Var}}
\newcommand{\Cov}{\text{Cov}}

\newcommand{\tmix}{t_{\mathrm{mix}}}

\newcommand{\gammaps}{\gamma_{\mathrm{ps}}}
\newcommand{\dtv}{d_{\mathrm{TV}}}
\newcommand{\mtx}[1]{\bm{#1}}

\newcommand{\blind}{0}

\begin{document}

\def\spacingset#1{\renewcommand{\baselinestretch}%
{#1}\small\normalsize} \spacingset{1}


\if0\blind
{
  \title{\bf Non-asymptotic confidence intervals for MCMC in practice}
  \author{Benjamin M. Gyori\hspace{.2cm}\\
    Department of Systems Biology, \\Harvard University\\
    and \\
    Daniel Paulin \\
    Department of Statistics and Applied Probability, \\National University of Singapore}
  \maketitle
} \fi

\if1\blind
{
  \bigskip
  \bigskip
  \bigskip
  \begin{center}
    {\LARGE\bf Non-asymptotic confidence intervals for MCMC in practice}
\end{center}
  \medskip
} \fi

\bigskip
\begin{abstract}
Using concentration inequalities, we give non-asymptotic confidence intervals for estimates obtained by Markov chain Monte Carlo (MCMC) simulations, when using the approximation 
 $\E_{\pi} f\approx (1/(N-t_0))\cdot \sum_{i=t_0+1}^N f(X_i)$. To allow the application of non-asymptotic error bounds in practice, here we state bounds formulated in terms of the spectral properties of the chain and the properties of $f$ and propose estimators of the parameters appearing in the bounds, including the spectral gap, mixing time, and asymptotic variance. We  introduce a method for setting the burn-in time and the initial distribution that is theoretically well-founded and yet is relatively simple to apply. We also investigate the estimation of $\E_{\pi}f$ via subsampling and by using parallel runs instead of a single run. Our results are applicable on both reversible and non-reversible Markov chains on discrete as well as general state spaces. We illustrate our methods by simulations for three examples of Bayesian inference in the context of risk models and clinical trials.
\end{abstract}

\noindent%

\vspace{10pt}

\noindent Source code available at: \url{http://github.com/bgyori/mcmc\_nonasym}

\vfill

\newpage
\spacingset{1.45} 

\input{introduction}

\input{inequalities}

\input{parameters}

\input{methodology}

\input{simulations}

\input{finalremarks}

\section*{Acknowledgements}
The authors thank Daniel Rudolf for his comments, and Lee Hwee Kuan for his contribution to the simulation code. DP was supported by AcRF Tier 2 grant R-155-000-143-112 and AcRF Tier 1 grant R-155-000-150-133. 

\bibliographystyle{chicago}

\bibliography{References}

\pagebreak
\input{appendix}

\end{document}

%% file: introduction.tex
\section{Introduction}
The Monte Carlo method relies on taking independent samples from a probability distribution to approximate an integral with respect to that distribution. Often, however, it is impossible or impractical to obtain such independent samples. One may still be able to construct a Markov chain whose stationary distribution matches the target distribution. It is then possible to obtain a series of dependent samples by sampling from the Markov chain. This method is called Markov chain Monte Carlo (MCMC). 

Let $X_1,X_2,\ldots,$ be a time-homogeneous, ergodic Markov chain, taking values in $\Omega$, and having stationary distribution $\pi$. Suppose that we are interested in computing $\E_{\pi} f$ for some $f:\Omega \to \R$. 
Then we usually make the approximation 
\begin{equation}\label{eqaverr}
\E_{\pi} f\approx \frac{\sum_{i=t_{0}+1}^N f(X_i)}{N-t_0},
\end{equation}
for some $t_0\ge 0$ (``burn-in time"). For $t_0$ fixed, and $N\to \infty$, this average converges to $\E_{\pi} f$ by the ergodic theorem. However, it is not clear how much this convergence is slowed down due to the dependence of the samples. Consequently, an important question in practice is, how large should $N$ be so that this approximation is correct to a certain level of precision?  To answer this question, we need to use error bounds.

Most of the error bounds in the literature are based on the convergence of the MCMC empirical average to the normal distribution. Although these bounds are convenient because they only depend on the asymptotic variance, due to the asymptotic nature of these bounds, they may underestimate the error for finite sample sizes. In contrast, concentration inequalities can provide finite sample size (non-asymptotic) error bounds. However, these error bounds have had limited practical applicability since parameters appearing in the inequalities (such as the variance, spectral gap and mixing time) are often not known in practice and principled methods for their estimation have not been available.

The main objective of this paper is to establish the practical applicability of non-asymptotic error bounds by giving estimators of their key parameters based on the data. We state estimators for the variance and the asymptotic variance ($\Var_{\pi}(f)$ and $\sigma^2$), and prove concentration inequalities to estimate the precision of these estimators. We state practically computable estimators for the spectral gap and the mixing time, and give a theoretical explanation for their usage. In addition to the estimators for the parameters, we also show how the error bounds are affected by subsampling (i.e.\ averaging only every $m$th term in \eqref{eqaverr}), and by running several parallel chains instead of a single chain. We explain when it is worthwhile to use these techniques.

We demonstrate our method on simulation results for Bayesian inference for risk models and clinical trial design. Our results show that the non-asymptotic error bounds based on concentration inequalities are more reliable than those based on the normal approximation. Specifically, we find that when the objective is to calculate the expected value of an indicator function (for instance an indicator of success or failure), the normal approximation often  misleadingly underestimates the error while non-asymptotic error bounds remain robust. This is an important consideration in sensitive settings. 

The structure of the paper is as follows. In the remainder of this section we explain basics from the theory of general state space Markov chains and briefly review the results in the literature on bounding the error of MCMC empirical averages. In Section \ref{sect:concentrationbounds}, we state non-asymptotic error bounds in a form that is practically applicable. In Section \ref{estparameters}, we propose estimators for the variance, asymptotic variance, spectral gap and mixing time of the chain used in the error bounds. In Section \ref{methodologicalcons}, we explain how error bounds are affected by subsampling and parallel runs. Finally, Section \ref{sectsimulations} contains our simulation results. The proof of the concentration bounds for our estimators of the variance and asymptotic variance and details on our simulation examples are presented in the Appendix. 

\subsection{Basic definitions for general state space Markov chains}
\label{section_preliminary}

In this section, we give some definitions from the theory of general state space Markov chains.
The total variational distance of two measures $P,Q$ defined on the measurable space $(\Omega,\F)$ is defined as
\begin{equation}\label{dtvdef2}\dtv(P,Q):=\inf_{\pi[X\sim P, Y\sim Q]} \E_{\pi}\left(\II[X\ne Y]\right),\end{equation}
here $\pi[X\sim P, Y\sim Q]$ denotes a distribution on $(\Omega^2,\F\times \F)$ with marginals $P$, and $Q$, and the infimum is taken over all such distributions.

We say that a Markov chain is \emph{$\phi$-irreducible}, if there exists a non-zero $\sigma$-finite measure $\phi$ on $\Omega$ such that for all $A\subset \Omega$ with $\phi(A)>0$, and for all $x\in \Omega$, there exists a positive integer $n=n(x,A)$ such that $P^n(x,A)>0$ (with $P^n$ denoting the $n$ step Markov kernel).

We call a Markov chain with stationary distribution $\pi$ \emph{aperiodic} if there do not exist $d\ge 2$, and disjoints measurable subsets $\Omega_1,\ldots,\Omega_d \subset \Omega$ with $\pi(\Omega_1)>0$, $P(x,\Omega_{i+1})=1$ for all $x\in \Omega_i$, $1\le i\le d-1$, and $P(x,\Omega_1)=1$ for all $x\in \Omega_d$.

These properties are sufficient for convergence to a stationary distribution. Theorem 4 of \cite{Robertsgeneral} shows that if a Markov chain on a state space with countably generated $\sigma$-algebra is $\phi$-irreducible and aperiodic, and has a stationary distribution $\pi(\cdot)$, then for $\pi$-almost every $x\in \Omega$, 
\[\lim_{n\to \infty}\dtv(P^n(x,\cdot),\pi(\cdot))=0.\]
A Markov chain with stationary distribution $\pi$, state space $\Omega$, and transition kernel $P(x,dy)$ is \emph{uniformly ergodic} if 
\[\sup_{x\in \Omega} \dtv\left(P^n(x,\cdot),\pi\right)\le M\rho^n, \hspace{ 5 mm} n=1,2,3,\ldots\]
for some $\rho<1$ and $M<\infty$, and we say that it is \emph{geometrically ergodic}, if
\[\dtv\left(P^n(x,\cdot),\pi\right)\le M(x)\rho^n, \hspace{ 5 mm} n=1,2,3,\ldots\]
for some $\rho<1$, where $M(x)<\infty$ for $\pi$-almost every $x\in \Omega$.

The mixing time of a time-homogeneous Markov chain with general state space is defined the following way (see Section 4.5 and 4.6 of \cite{peresbook}).
\begin{definition}\label{mixhom}
Let $X_1,X_2,\ldots$ be a time-homogeneous Markov chain with transition kernel $P(x,dy)$, state space $\Omega$ (a Polish space), and stationary distribution $\pi$.
The \emph{mixing time} of the chain, denoted by $\tmix$, is defined as
\[d(t):=\sup_{x\in \Omega} \dtv\left(P^t(x,\cdot),\pi \right), \text{ and }\hspace{2mm}\tmix:=\min\{t: d(t)\le 1/4\}.\]
\end{definition}

The fact that $\tmix$ is finite is equivalent to the \emph{uniform ergodicity} of the chain (see \cite{Robertsgeneral}, Section 3.3). We call a Markov chain $X_1,X_2,\ldots$ on state space $(\Omega,\F)$ with transition kernel $P(x,dy)$ \emph{reversible} if there exists a probability measure $\pi$ on $(\Omega,\F)$ satisfying the detailed balance conditions,
\begin{equation}
\pi(dx) P(x,dy)= \pi(dy) P(y,dx) \text{ for  every } x,y\in \Omega.
\end{equation}
Define $L_2(\pi)$ as the Hilbert space of complex valued measurable functions that are square integrable with respect to $\pi$, endowed with the inner product $(f,g)=\int f g^*\, \mathrm{d}\pi$. $P$ can be then viewed as a linear operator on $L_2(\pi)$, denoted by $\mtx{P}$, defined as \[(\mtx{P} f)(x):=\E_{P(x,\cdot)}(f),\] and reversibility is equivalent to the self-adjointness of $\mtx{P}$.
The operator $P$ acts on measures to the left, i.e.\ for every measurable subset $A$ of $\Omega$, 
\[(\mu \mtx{P})(A):=\int_{x\in \Omega} P(x,A) \mu(\mathrm{d} x).\]
For a Markov chain with stationary distribution $\pi$, we define the \emph{spectrum} of the chain as
\[S_2:=\{\lambda\in \C\setminus 0: (\lambda\mtx{I}-\mtx{P})^{-1}\text{ does not exist as a bounded linear operator on } L_2(\pi)\}.\] For reversible chains, $S_2$ lies on the real line. 
Now we define the spectral properties of the chain, the spectral gap, absolute spectral gap, and the pseudo spectral gap.
\begin{definition}\label{spectralgapdef}
We define the \emph{spectral gap} for reversible chains as
\begin{align*}
\gamma&:=1-\sup\{\lambda: \lambda\in S_2, \lambda\ne 1\}  \quad \text{if eigenvalue 1 has multiplicity 1,}\\
\gamma&:=0 \quad \text{otherwise}.
\end{align*}
For both reversible, and non-reversible chains, we define the \emph{absolute spectral gap} as
\begin{align*}
\gamma^*&:=1-\sup\{|\lambda|: \lambda\in S_2, \lambda\ne 1\}  \quad \text{if eigenvalue 1 has multiplicity 1,}\\
\gamma^*&:=0 \quad \text{otherwise}.
\end{align*}
For both reversible, and non-reversible chains, we define the \emph{pseudo spectral gap} as
\begin{equation}\label{gammapsdef}
\gammaps:=\max_{k\ge 1} \left\{\gamma((\mtx{P}^*)^k \mtx{P}^k)/k\right\}, 
\end{equation}
where $P^*$ is the adjoint of $P$ in $L_2(\pi)$, and $\gamma((\mtx{P}^*)^k \mtx{P}^k)$ denotes the spectral gap of the self-adjoint operator $(\mtx{P}^*)^k \mtx{P}^k$.
\end{definition}
\begin{remark}
Note that for reversible chains, $\gamma\ge \gamma^*$.
\end{remark}

The relation between spectral gap, pseudo spectral gap, and mixing time is given by the following proposition (Propositions 3.3 and 3.4 of \cite{Martoncoupling}).
\begin{proposition}\label{tmixlambdaprop}
For uniformly ergodic reversible/non-reversible chains, respectively,
\begin{equation}\label{tmixbound1}
\gamma^*\ge \frac{1}{1+\tmix/\log(2)}, \hspace{5mm} \gammaps\ge \frac{1}{2\tmix}.
\end{equation}

For reversible/non-reversible chains on finite state spaces, respectively,
\begin{equation}
\tmix\le \frac{2\log(2) + \log(1/\pi_{\min})}{2\gamma^*}, \hspace{5mm}\tmix\le \frac{1+2\log(2) + \log(1/\pi_{\min})}{\gammaps},
\end{equation}
with $\pi_{\min}=\min_{x\in \Omega} \pi(x)$.
\end{proposition}
\begin{remark}
This proposition means that for Markov chains on finite state spaces, fast mixing and large spectral gap (or pseudo spectral gap) are essentially equivalent.
\end{remark}

\subsection{Previous results on the error of MCMC averages}\label{SecComparison}
In this section we briefly review two approaches to estimating the errors of MCMC empirical averages: the central limit theorem and concentration inequalities. We note that there are some other convergence diagnostic methods in the literature for choosing the ``burn-in time'', such as the Gelman-Rubin diagnostic (see \cite{gelman1992inference},  \cite{cowles1996markov}, \cite{HandbookofMCMC}). However, we are unaware of  generally applicable methods with rigorous theoretical justification.

\subsubsection{Error estimation by the central limit theorem}
The following central limit theorem (CLT) for Markov chains is stated as in \cite{Robertsgeneral} (Theorems 23 and 24, and Proposition 29).
\begin{thm}
Let $(X_i)_{i\ge 1}$ be a Markov chain taking values in some general state space $\Omega$, with stationary distribution $\pi$. Let $f:\Omega\to \R$, with $\E_{\pi}(f^2)<\infty$.
Define the asymptotic variance of $f$, denoted by $\sigma^2$ as
\begin{equation}\label{sigma2vardef}\sigma^2:=\lim_{n\to \infty}\frac{1}{n}\Var_{\pi}(f(X_1)+\ldots+f(X_n)).
\end{equation}

Let $Z_n:=\frac{1}{n}\sum_{i=1}^n f(X_i)$, then
if $(X_i)_{i\ge 1}$ is uniformly ergodic, then for $\pi$-almost every starting point $x$ (i.e.\ $X_1=x$),
\begin{equation}
\sqrt{n}(Z_n-\E_{\pi}f)\Rightarrow N(0,\sigma^2).
\end{equation}
The same holds if $(X_i)_{i\ge 1}$ is geometrically ergodic, and $\E_{\pi}(f^{2+\delta})<\infty$ for some $\delta>0$.
\end{thm}
\begin{remark}
For reversible/non-reversible chains, respectively, 
\[\sigma^2\le 2\Var_{\pi}(f)/\gamma, \hspace{2mm} \sigma^2\le 4\Var_{\pi}(f)/\gammaps,\]
see \cite{Martoncoupling}, Theorems 3.1 and 3.2.
\end{remark}

To make use of the limiting distribution $N(0,\sigma^2)$, one needs an estimator of $\sigma^2$ from the data. There are many such estimators in the literature, see \cite{geyer1992practical}, \cite{robertconvergence},  \cite{Hobert}, \cite{JonesHaran}, \cite{BednorzLatusz}, \cite{FlegalJones}.
Based on the regeneration properties of the chain, these estimators are shown to be asymptotically consistent under very mild conditions. However, we are not aware of non-asymptotic bounds on the precision of these estimators. In Section \ref{estparameters} we will define an estimator of the asymptotic variance, $\hat{\sigma}^2$, show its consistency for bounded functions, and prove non-asymptotic error bounds for it in terms of the mixing time of the chain.

Compared to Bernstein-type concentration inequalities, the CLT approach can be slightly sharper for small deviations from the mean. However, the normal approximation is only true asymptotically, whereas concentration bounds hold for any sample size. 
\cite{lezaud1998etude} proves a Berry-Ess\'een bound for Markov chains, estimating the quality of the normal approximation. Unfortunately the constants are too large for practical applicability.

\subsubsection{Error estimation by concentration inequalities}
The simplest way of obtaining concentration is by variance bounds, which lead to quadratic tail bounds via Chebyshev's inequality. We have included two such bounds in Section \ref{sect:concentrationbounds}, Theorem \ref{Chebrevthm} and Theorem \ref{Chebnonrevthm}. The advantage of these results is that they work even for unbounded functions $f$ that have a finite variance (with respect to $\pi$). Their disadvantage is that they only imply quadratic decay instead of the Gaussian decay of concentration inequalities. Various non-asymptotic results also exist for the mean square error of the MCMC estimate. 
A similar bound on the mean square error using the spectral gap and $\|f\|_p$ for $p \ge 2$ is given in \cite{rudolf2011explicit}, and \cite{latuszynski2011nonasymptotic} gives an asymptotically sharp bound on the MSE as a function of the asymptotic variance $\sigma^2$. 

Hoeffding-type  concentration inequalities for empirical averages of Markov chains were first proven by \cite{gillman1998chernoff}, using spectral methods. This was further developed in the seminal work \cite{lezaud1998etude} (see also \cite{Lezaud1}), which has shown Bernstein-type inequalities for reversible and non-reversible chains with general state spaces. 
In \cite{Martoncoupling}, we have shown improved  Bernstein-type inequalities for reversible and non-reversible chains. A sharp version of the Hoeffding bound for reversible finite state space chains was proven in \cite{leon2004optimal}, and generalized to general state spaces, and possibly non-reversible chains in \cite{MiasojedowHoeffding}.

\cite{ollivier2010} proves Bernstein-type concentration inequalities for MCMC empirical averages of Lipschitz functions in metric spaces, based on a contraction coefficient of the Markov chain with respect to the metric. \cite{Mixingandconcentration} shows some generalizations.

Hoeffding bounds under different, regeneration-type assumptions on the chain were proven in \cite{GlynnOrmoneit}, and in \cite{MoulinesMLE}. Bernstein inequalities were proven under such assumptions in \cite{Adamczaktail}. These inequalities are more generally applicable than our bounds based on spectral methods (for example to chains that are not geometrically ergodic), but they are more complicated and involve many terms that are hard to estimate in practice.

%% file: inequalities.tex
\section{Concentration bounds}
\label{sect:concentrationbounds}
We make the following assumption in all of the theorems stated in this paper (which we state here to avoid unnecessary repetition).
\begin{assumption}
We always assume that the Markov chain $X_1, X_2, \ldots$ is time homogeneous, $\phi$-irreducible, and aperiodic. We also assume that it has a Polish state space $\Omega$, Markov kernel $P(x,\mathrm{d}y)$, and denote its unique stationary distribution by $\pi$.
\end{assumption}
\noindent In the following sections, we will state results about the empirical average defined as
\begin{equation}\label{Zdefeq}
Z:=\left(\sum_{i=t_0+1}^N f(X_i)\right)/(N-t_0).
\end{equation}
The quantity $t_0$, the so called \emph{burn-in time} corresponds to the number of samples discarded from the beginning of the chain. In the rest of this section, we first state some results about choosing the necessary burn-in periods, and then present concentration inequalities that give non-asymptotic bounds on the approximation \eqref{eqaverr}. We state Chebyshev and Bernstein-type inequalities for both reversible and non-reversible chains. 

\subsection{Setting the burn-in time}
For initial distribution $q$, and stationary distribution $\pi$, let
\begin{equation}E(t_0):= \dtv\left(q \mtx{P}^{t_0}, \pi\right).
\end{equation}
This definition will be useful for stating our bounds. For uniformly ergodic Markov chains (including ergodic finite state chains), by Proposition 3.11 of \cite{Martoncoupling}, we have
\begin{equation}
E(t_0)\le 2^{-\lfloor n/\tmix \rfloor}.
\end{equation}
Markov chains on general state spaces are often not uniformly ergodic. In such cases,  $E(t_0)$ can be bounded using a different approach.
If a distribution $q$ on $(\Omega,\F)$  is absolutely continuous with respect to $\pi$, we define the \emph{$\chi^2$ contrast} of $q$ and $\pi$, denoted by $N_q$, as
\begin{equation}\label{Nqdef}
N_q:=\E_{\pi}\left(\left(\frac{\mathrm{d}\, q}{\mathrm{d}\, \pi}\right)^2\right)=\int_{x\in \Omega} \frac{\mathrm{d}\, q}{\mathrm{d}\, \pi}(x) q(\mathrm{d} x).
\end{equation}
If $q$ is not absolutely continuous with respect to $\pi$, then we define $N_q:=\infty$. Using $N_q$, by Proposition 3.11 of \cite{Martoncoupling}, for reversible/non-reversible chains, respectively,
\begin{equation}\label{Et0Nqbound}
E(t_0)\le \frac{1}{2}(1-\gamma^*)^{t_0} \cdot \sqrt{N_q-1}, \hspace{2mm} E(t_0)\le \frac{1}{2}(1-\gammaps)^{(t_0-1/\gammaps)/2} \cdot \sqrt{N_q-1}.
\end{equation}
If starting from a fixed point $x$, $q(\{x\})=1$, and $N_q=1/\pi(\{x\})$. Note that if $\pi(\{x\})=0$, then $N_q=\infty$, however this can be usually remedied by taking a few steps in the Markov chain and looking at $N_{\mtx{P}^{k}(x,\cdot)}$ instead (as long as $P^{k}(x,\cdot)$ is absolutely continuous with respect to $\pi$ for every $x\in \Omega$). 

In addition to this, we now propose a new, more generally applicable approach for bounding $N_q$. Suppose that we start from a initial distribution $q$ on a measurable set $\Omega_0\subset \Omega$, such that $\pi(\Omega_0)>0$ and $q$ is absolutely continuous with respect to $\pi$,  (for example, if $\Omega=\R^d$, and $\pi$ has a continuous density with respect to the Lebesgue measure, then we can choose $\Omega_0$ to be a $d$ dimensional box, and $q$ be the uniform distribution on $\Omega_0$). Then the last equation in \eqref{Nqdef} implies that
\begin{equation}\label{Nqbound}
N_q\le \sup_{x\in \Omega_0}\frac{\mathrm{d}\, q}{\mathrm{d}\, \pi}(x)\le \frac{\sup_{x\in \Omega_0}\frac{\mathrm{d}\, q}{\mathrm{d}\, \pi}(x)}{\inf_{x\in \Omega_0}\frac{\mathrm{d}\, q}{\mathrm{d}\, \pi}(x)}\cdot \frac{1}{\pi(\Omega_0)}=\frac{\sup_{x\in \Omega_0}\frac{\mathrm{d}\, \pi}{\mathrm{d}\, q}(x)}{\inf_{x\in \Omega_0}\frac{\mathrm{d}\, \pi}{\mathrm{d}\, q}(x)}\cdot \frac{1}{\pi(\Omega_0)},
\end{equation}
where we have used the fact that $\inf_{x\in \Omega_0}\frac{\mathrm{d}\, q}{\mathrm{d}\, \pi}(x)\le 1/\pi(\Omega_0)$. Notice that the ratio $\frac{\sup_{x\in \Omega_0}\frac{\mathrm{d}\, \pi}{\mathrm{d}\, q}(x)}{\inf_{x\in \Omega_0}\frac{\mathrm{d}\, \pi}{\mathrm{d}\, q}(x)}$ can be bounded even if the normalizing constant in the definition of $\pi(x)$ is not known, which is typical in practice. Finally,  $\pi(\Omega_0)$ can be lower bounded by bounding the tails of the distribution $\pi$, and showing that they become small far away from the center.

\subsection{Reversible chains}
First, we state a Chebyshev-type inequality based on Theorem 3.1 of \cite{Martoncoupling}.
\begin{thm}[Chebyshev inequality for reversible Markov chains]\label{Chebrevthm}
Let $X_1,\ldots,X_N$ be a reversible Markov chain with spectral gap $\gamma$.
Let $f$ be a measurable function from $\Omega$ to $\R$, satisfying that $\E_{\pi} \left(f^2\right)<\infty$.
Denote the variance of $f$ by $V_f:=\Var_{\pi}(f)$, and let $\sigma^2$ be the asymptotic variance of $f$, defined in \eqref{sigma2vardef}.

If we start from the stationary distribution, then 
\begin{equation}\label{varempboundrev}
|\Var_{\pi}\left[(f(X_1)+\ldots+f(X_N))/N\right]-\sigma^2/N|\le \frac{4 V_f}{\gamma^2}\cdot \frac{1}{N^2}.
\end{equation}
Let $Z$ be as in \eqref{Zdefeq}. By Chebyshev's inequality, it follows that for any initial distribution $q$, for any $t\ge 0$, we have
\begin{equation}\label{chebempboundrev}\PP_{q}\left[|Z-E_{\pi} f|\ge t\right]\le \frac{\sigma^2+4 V_f/((N-t_0)\gamma^2)}{(N-t_0) t^2}+E(t_0).\end{equation}
\end{thm}
Now we state a Bernstein-type result based on Theorem 3.3 of \cite{Martoncoupling}.
\begin{thm}[Bernstein inequality for reversible Markov chains]\label{thmBernsteinrev}
Let $X_1,$ $\ldots,X_N$ be a reversible Markov chain with spectral gap $\gamma>0$. Suppose that 
$f:\Omega \to \R$ satisfies that $\sup_{x\in \Omega}|f(x)-\E_{\pi} f|\le C$ for some finite $C$. Let $V_f$ and $\sigma^2$ be defined as in Theorem \ref{Chebrevthm}, and $Z$ as in \eqref{Zdefeq}.
Then for any initial distribution $q$, for any $t\ge 0$, we have
\begin{equation}\label{ReversibleBernsteineq1}
\PP_q\left[ |Z-\E_{\pi} f| \ge t \right]\le 2\exp\left[-\frac{(N-t_0)t^2}{2(\sigma^2+0.8V_f) + 10tC/\gamma} \right]+ E(t_0).
\end{equation}
\end{thm}

\subsection{Non-reversible chains}
The most popular MCMC methods use reversible chains, in particular, the Metropolis-Hastings algorithm and the Glauber dynamics (with random scan) are reversible. On the other hand, the Glauber dynamics chain with systematic scan is non-reversible. In fact, using non-reversible chains instead of reversible ones can speed up the mixing time in some cases (see
\cite{diaconisnonrev}, and \cite{Lovaszlifting}). Therefore it is of interest to show concentration bounds for non-reversible Markov chains too. The following result is a Chebyshev-type inequality based on Theorem 3.2 of \cite{Martoncoupling}.
\begin{thm}[Chebyshev inequality for non-reversible Markov chains]\label{Chebnonrevthm}
Let $X_1, \ldots, X_N$ be a  Markov chain with pseudo spectral gap $\gammaps$.
Let $f$ be a measurable function from $\Omega$ to $\R$, satisfying that $\E_{\pi} f^2<\infty$.
Let $V_f$ and $\sigma^2$ be defined as in Theorem \ref{Chebrevthm}. If we start from the stationary distribution, then 
\begin{equation}\label{varempboundnonrev}
|\Var_{\pi}\left[(f(X_1)+\ldots+f(X_N))/N\right]-\sigma^2/N|\le \frac{16 V_f}{\gammaps^2}\cdot \frac{1}{N^2}.
\end{equation}

Let $Z$ be as in \eqref{Zdefeq}, then by Chebyshev's inequality, for any initial distribution $q$, for any $t\ge 0$, we have
\begin{equation}\label{chebempboundnonrev}
\PP_{q}\left[|Z-\E_{\pi} f|\ge t\right]\le \frac{\sigma^2+16V_f/((N-t_0)\gammaps^2)}{(N-t_0) t^2}+E(t_0).
\end{equation}
\end{thm}
Now we state a Bernstein-type inequality based on Theorem 3.4 of \cite{Martoncoupling}.
\begin{thm}[Bernstein inequality for non-reversible Markov chains]\label{NonReversibleBernsteinthmpseudo}
Let $X_1,\ldots, X_N$ be a Markov chain with pseudo spectral gap $\gammaps$. 
Suppose that $f:\Omega \to \R$ satisfies that $\sup_{x\in \Omega}|f(x)-\E_{\pi} f|\le C$ for some finite $C$. Let $V_f$ and $\sigma^2$ be defined as in Theorem \ref{Chebrevthm}, and $Z$ as in \eqref{Zdefeq}. Then for any initial distribution $q$, for any $t\ge 0$, we have
\begin{equation}\label{NonReversibleBernsteineq1}
\PP_q\left[ |Z-\E_{\pi} f| \ge t \right]\le 2\exp\left[-\frac{(N-t_0-1/\gammaps)t^2 \gammaps}{8 V_f +20 C t} \right]+ E(t_0).
\end{equation}
\end{thm}
\begin{remark}
An important assumption of the Bernstein-type inequalities is the boundedness of $f$. Note that via a simple truncation argument, the results can be also extended to unbounded functions, for more details, see Proposition 3.12 of \cite{Martoncoupling}.
\end{remark}

%% file: parameters.tex
\section{Estimation of parameters in practice} 
\label{estparameters}
The main difficulty we encounter when applying our inequalities is that, in general, we do not know $V_f = \Var_{\pi}(f)$ and $\sigma^2$ (see \eqref{sigma2vardef}). In many cases, the spectral gap $\gamma$, pseudo spectral gap $\gammaps$, and mixing time $\tmix$ are also unknown.

In the next two sections, we are going to give estimates to these quantities based on an initial sample $f(X_{\hat{t}_{0}+1}),\ldots,f(X_{\hat{N}})$.

\subsection{Estimation of the variance and the asymptotic variance}
From the definitions, it is easy to see that we can estimate $V_f$ as
\begin{equation}\label{Vfhatdef}
\hat{V}_f:= \frac{1}{\hat{N}-\hat{t}_0}\left(\sum_{i=\hat{t}_0+1}^{\hat{N}} f^2(X_i)\right) - \left(\frac{1}{\hat{N}-\hat{t}_0}\sum_{i=\hat{t}_0+1}^{\hat{N}} f(X_i)\right)^2.
\end{equation}
Note that the consistency of this estimator under very mild assumptions follows from the strong law of large numbers for Markov chains (see Theorem 17.0.1 of \cite{MeynTweedie}). The next proposition gives a non-asymptotic bound on the upper tails of $V_f-\hat{V}_f$ for uniformly ergodic Markov chains (the proof can be found in the Appendix).
\begin{proposition}\label{propVhatV}
Suppose that $X_1,\ldots,X_{\hat{N}}$ is an uniformly ergodic Markov chain, with stationary distribution $\pi$, and initial distribution $q$. Suppose that $f:\Omega\to \R$ satisfies that $\sup_{x\in \Omega}|f(x)-\E_{\pi} f|\le C$ for some finite $C$. Then for any $T \ge 0$,
\begin{equation}\label{eqVfhatV}
\PP_q\left(V_f-\hat{V}_f\ge \frac{8\tmix}{\hat{N}-\hat{t}_0}+T\right)\le \exp\left( \frac{-(\hat{N}-\hat{t}_0) T^2}{200 C^4 \tmix}\right)+E(\hat{t}_0).
\end{equation}
\end{proposition}
Now we propose an estimator to the asymptotic variance $\sigma^2$. For some integer $k\in [1,\hat{N}-\hat{t}_0-1]$, let
\begin{equation}
\hat{\sigma}^2(k):=\left(\hat{\rho}_0+2\sum_{i=1}^{k}\hat{\rho}_i\right)\cdot \frac{\hat{N}-\hat{t}_0-k}{\hat{N}-\hat{t}_0-3k-1},
\end{equation}
with 
\begin{equation}
\hat{\rho}_i:=\frac{\sum_{j=\hat{t}_0+1}^{\hat{N}-k} f(X_j)f(X_{j+i})}{\hat{N}-\hat{t}_0-k}
-\frac{1}{2}\left(\frac{\sum_{j=\hat{t}_0+1}^{\hat{N}-k} f(X_j)}{\hat{N}-\hat{t}_0-k}\right)^2
-\frac{1}{2}\left(\frac{\sum_{j=\hat{t}_0+i+1}^{\hat{N}-k+i} f(X_j)}{\hat{N}-\hat{t}_0-k}\right)^2.
\end{equation}
If we make the choice $k=\lfloor c\cdot N^{1/3}\rfloor$ for a positive constant $c$, the estimator can be shown to be consistent under very mild conditions, see Theorem 1 of \cite{FlegalJones}. 

The following two propositions bounds on the bias of $\hat{\sigma}^2(k)$, and state a non-asymptotic error bound for it, for uniformly ergodic Markov chains (see the Appendix for proofs).
\begin{proposition}\label{sigma2kbias}
For stationary, reversible chains, when $k$ is even, the expected value of $\hat{\sigma}^2(k)$ satisfies the following inequality:
\begin{equation}\label{hatsigmakrevexpectationbound}
-L_k\le \sigma^2-\E_{\pi}(\hat{\sigma}^2(k))\le U_k,
\end{equation}
with
\begin{align*}
L_k&:=\left(\min\left(V_f,\frac{2V_f}{\gamma}(1-\gamma^*)^{k+1}\right)+\frac{4V_f}{\gamma^2}\frac{2k+1}{(\hat{N}-\hat{t}_0-k)^2}\right)\cdot \frac{\hat{N}-\hat{t}_0-k}{\hat{N}-\hat{t}_0-3k-1} , \text{ and}\\
U_k&:=\left(\frac{2V_f}{\gamma}(1-\min(\gamma,1))^{k+1}+\frac{4V_f}{\gamma^2}\frac{2k+1}{(\hat{N}-\hat{t}_0-k)^2}\right)\cdot \frac{\hat{N}-\hat{t}_0-k}{\hat{N}-\hat{t}_0-3k-1}.
\end{align*}
For stationary non-reversible chains, for any $k\ge 1$,
\begin{equation}\label{hatsigmaknonrevexpectationbound}
|\E_{\pi}(\hat{\sigma}^2(k))-\sigma^2|\le W_k,
\end{equation}
with 
\[W_k:=\frac{4V_f}{\gammaps}(1-\gammaps)^{(k+1-1/\gammaps)/2}+\frac{16V_f}{\gammaps^2}\frac{2k+1}{(\hat{N}-\hat{t}_0-k)^2}.\]
\end{proposition}
\begin{proposition}\label{sigma2kerrorbound}
Suppose that $f:\Omega\to \R$ satisfies that $\sup_{x\in \Omega}|f(x)-\E_{\pi} f|\le C$ for some finite $C$. In the case of stationary, uniformly ergodic chains, we have for any $t\ge 0$,
\begin{equation}\label{eqconcsigmak}
\PP_{\pi}(|\hat{\sigma}^2(k)-\E_{\pi}(\hat{\sigma}^2(k))|\ge t)\le 2\exp\left(\frac{-t^2(\hat{N}-\hat{t}_0-3k-1)}{512(2k+1)^2C^4 \tmix}\right),
\end{equation}
This implies that for uniformly ergodic reversible chains, with arbitrary initial distribution $q$, for even $k\ge 2$, any $t\ge 0$,
\begin{equation}\label{eqconcsigmakrev}
\PP_{q}\left(\sigma^2-\hat{\sigma}^2(k)\ge U_k+t\right)
\le \exp\left(\frac{-t^2(\hat{N}-\hat{t}_0-3k-1)}{512(2k+1)^2C^4\tmix}\right)+E(\hat{t}_0),
\end{equation}
and for uniformly ergodic non-reversible chains, for any $k\ge 1$, $t\ge 0$,
\begin{equation}\label{eqconcsigmaknonrev}
\PP_{q}\left(\sigma^2-\hat{\sigma}^2(k)\ge W_k+t\right)
\le \exp\left(\frac{-t^2(\hat{N}-\hat{t}_0-3k-1)}{512(2k+1)^2C^4\tmix}\right)+E(\hat{t}_0).
\end{equation}
\end{proposition}
\begin{remark}
It is clear that if we increase $k$, the bias 
$|\sigma^2-\E_{\pi}(\hat{\sigma}^2(k))|$ becomes smaller, but the concentration bounds become weaker. With the choice
\begin{equation}\label{hatsigma2def}
\hat{t}_0:=\lfloor 0.1 \hat{N}\rfloor, \hspace{2mm} k:=10\cdot \left\lfloor \hat{N}^{1/3}\right\rfloor, \hspace{2mm} \hat{\sigma}^2:=\hat{\sigma}^2(k),
\end{equation}
our bounds imply that for bounded functions, $\hat{\sigma}^2$ will be a consistent estimate of $\sigma^2$ as $\hat{N}\to \infty$, for any uniformly ergodic Markov chain, irrespectively of the value of the mixing time. This consistency also holds under very mild conditions, without assuming uniform ergodicity, see Theorem 1 of \cite{FlegalJones}. 
\end{remark}

\subsection{Estimation of the spectral gap and the mixing time}\label{sectionestimatespectralgap}
In this section, we will state some estimators for the spectral gap, pseudo spectral gap, and mixing time of the Markov chain.

A key parameter appearing in the concentration inequalities for reversible chains is the \emph{spectral gap} of the Markov chain, denoted $\gamma$. 
It is known that for reversible chains, the existence of spectral gap ($\gamma>0$) is equivalent to 
the geometric ergodicity of the chain (see \cite{Robertsgeneral}). Now we briefly review an iterative method for estimating the spectral gap, introduced in \cite{gyori2014hypothesis}.
For simplicity, we are going to assume that $\Omega=\R^{d}$ (but the method can also be simply adapted to other state spaces as well).  The method is based on the fact that under mild conditions on the function $f:\Omega\to \R$, 
\[\lim_{k\to \infty}(\Cov_{X_0\sim \pi}(f(X_0),f(X_k))/\Var_{\pi}(f))^{1/k}=1-\gamma^*,\]
with $\gamma^*\le \gamma$ denoting the absolute spectral gap. We do the following steps.
\begin{enumerate}
\item Run an initial simulation  of length $n$ yielding  values $X_1,\ldots, X_n$. In every step $1\le i\le n$, save each component $X_i^{1},\ldots, X_i^{d}$.
\item Set $\eta=1$, and for every $1\le k\le d$, compute
\begin{equation}\label{eq:gammahat}
\hat{\gamma}_{\eta, k} := 1-(\hat{\rho}_{\eta,k}/\hat{V}_{k})^{1/\eta},
\end{equation}
where
\begin{align}
\hat{V}_k&:= \frac{1}{n} \sum_{i=1}^{n}X_{i}^{k} - \left(\frac{1}{n} \sum_{i=1}^n X_{i}^{k}\right)^2 \\
\hat{\rho}_{\eta, k}(f) &:= \frac{1}{n-\eta}\sum_{i=1}^{n-\eta}\left(X_{i}^{k}-
\frac{1}{n-\eta}\sum_{j=1}^{n-\eta}X_j^{k}
\right)\left(X_{i+\eta}^{k}-
\frac{1}{n-\eta}\sum_{j=1}^{n-\eta}X_{j+\eta}^{k}
\right).
\end{align}

We call the minimum of $\hat{\gamma}_{\eta, 1},\ldots, \hat{\gamma}_{\eta, d}$ by $\hat{\gamma}_{\min}(1)$, and compute 
\begin{equation}\label{eq:eta1def}\eta(1):=\frac{\log(n\hat{\gamma}_{\min}(1))}{4\log(1/(1-\hat{\gamma}_{\min}(1)))}.\end{equation}
\item Inductively suppose that we have already computed $\eta(j)$ for $j\ge 1$, using \eqref{eq:eta1def}. Now compute $\gamma_{\min}(j+1)$ based on \eqref{eq:gammahat} using $\eta=\eta(j)$. If $\hat{\gamma}_{\min}(j+1)\ge \hat{\gamma}_{\min}(j)$, stop, and let $\hat{\gamma}:=\hat{\gamma}_{\min}(j)$. Otherwise compute $\eta(j+1)$ and repeat this step.
\item To make sure that there is a sufficient amount of initial data, accept the estimate only if $n$ satisfies $n>100/\hat{\gamma}$, otherwise choose $n=200/\hat{\gamma}$ and restart from Step 2.
\end{enumerate}
For non-reversible chains, the \emph{pseudo spectral gap} $\gammaps:=\max_{k\ge 1}\gamma((P^*)^k P^k)/k$ can be estimated as follows. We define our estimate as
\begin{equation}\label{eq:hatgammapseq}\hat{\gamma}_{\mathrm{ps}}=\max_{k\ge 1}\hat{\gamma}((P^*)^{2^k} P^{2^k})/2^k,\end{equation}
where for each $k$, we use the estimate $\hat{\gamma}$ from the previous paragraph (so that we run a Markov chain with kernel $(P^*)^{2^k} P^{2^k}$). We only use the powers of 2 for faster computation. By the definition of $\gammaps$, this estimate is expected to be less than or equal to $\gammaps$.
We do not need infinitely many evaluations in $k$ for computing \eqref{eq:hatgammapseq}, since $\gamma((P^*)^{2^k} P^{2^k})/2^k\le 2^{-k}$, so we can stop whenever $2^{-k}$ gets smaller than the maximum of the first $k-1$ terms. 

In our simulations, we have found that for geometrically ergodic chains that mix well, the estimators $\hat{\gamma}$ and $\hat{\gamma}_{\mathrm{ps}}$ converge rather fast even for relatively small sample sizes. If the chain is not geometrically ergodic, then these estimators tend to be very small, and they do converge to $0$ as the sample size increases. Geometric ergodicity can be typically ensured by restricting the state space to a compact set.

Based on Proposition \ref{tmixlambdaprop}, for finite state Markov chains, $\tmix$ can be bounded as
\begin{align*}
\tmix\le \frac{2\log(2) + \log(1/\pi_{\min})}{2\gamma^*}, \quad \tmix\le \frac{1+2\log(2) + \log(1/\pi_{\min})}{\gammaps},
\end{align*}
with $\pi_{\min}=\min_{x\in \Omega}\pi(x)$. In practice, it is usually easy to get lower bounds on $\pi_{\min}$, for example, if $\pi(x)\sim \exp(h(x))$, then $\pi_{\min}\ge \frac{1}{|\Omega|}\cdot \exp(-(\max_{x\in \Omega} {h(x)}-\min_{x\in \Omega} {h(x)}))$. Assuming that $\ul{\pi_{\min}}\le \pi_{\min}$ is a lower bound for $\pi_{\min}$, we propose the estimators
\begin{align}
\hat{t}_{\mathrm{mix}}^{\text{ finite state}}&:= \frac{2\log(2) + \log(1/\ul{\pi_{\min}})}{2\hat{\gamma}} \text{ for reversible chains, and }\\
\hat{t}_{\mathrm{mix}}^{\text{ finite state}}&:=\frac{1+2\log(2) + \log(1/\ul{\pi_{\min}})}{\hat{\gamma}_{\mathrm{ps}}}\text{ for non-reversible chains}.
\end{align}
For general state Markov chains, we need to assume uniform ergodicity for the mixing time to be finite. This can be usually achieved by restricting the state space $\Omega$ to a compact set, and forbidding moving outside of this set. The change in $\E_{\pi}(f)$ by this modification is typically easy to bound by bounding the tails of the distribution $\pi$.

Now we propose a new estimate of the mixing time for general state space Markov chains. Suppose that when starting from $x$, with probability $1-a(x)$, we stay in place, and with probability $a(x)$, we move according to some proposal distribution $q(x,\cdot)$ (this is typical for Metropolis-Hastings chains). Then the Markov kernel equals
$P(x,\cdot)=a(x) q(x,\cdot)+ (1-a(x))\delta_x$. Suppose that $q(x)$ is supported on some measurable set $\Omega_0^{x}\subset \Omega$,
and that both $\pi$ and $q(x,\cdot)$ are absolutely continuous with respect to a distribution $\mu$ on $\Omega$ (typically chosen as the uniform distribution). Then based on \eqref{Et0Nqbound} and 
\eqref{Nqbound} it is straightforward to show (see Section \ref{secmixingtimebound} of the Appendix for details) that for reversible chains,
\begin{align}
&\label{eqtmixboundgen}\tmix\le 1+\frac{2}{\inf_{x\in \Omega}a(x)}+\frac{2}{\gamma}+\frac{1}{2 \gamma}\cdot\\
\nonumber&\cdot\sup_{x\in \Omega}\left(\sup_{y\in \Omega_0^{x}}\log\frac{\mathrm{d}\, \pi}{\mathrm{d}\, \mu}(x)-\inf_{y\in \Omega_0^{x}}\log\frac{\mathrm{d}\, \pi}{\mathrm{d}\, \mu}(x) +\sup_{y\in \Omega_0^{x}}\log\frac{\mathrm{d}\, q}{\mathrm{d}\, \mu}(x)-\inf_{y\in \Omega_0^{x}}\log\frac{\mathrm{d}\, q}{\mathrm{d}\, \mu}(x) +\log(1/\pi(\Omega_0^{x}))\right),
\end{align}
and the same bound holds for non-reversible chains as well when $\gamma$ is replaced by $\gammaps/2$. Based on this, for reversible, general state space chains, we propose the estimate
\begin{align}\label{eqtmixestgeneral}&\hat{t}_{\mathrm{mix}}:= 1+\frac{2}{\inf_{x\in \Omega}a(x)}+\frac{2}{\hat{\gamma}}+\frac{1}{2\hat{\gamma}}\cdot\\
\nonumber&\cdot\sup_{x\in \Omega}\left(\sup_{y\in \Omega_0^{x}}\log\frac{\mathrm{d}\, \pi}{\mathrm{d}\, \mu}(x)-\inf_{y\in \Omega_0^{x}}\log\frac{\mathrm{d}\, \pi}{\mathrm{d}\, \mu}(x) +\sup_{y\in \Omega_0^{x}}\log\frac{\mathrm{d}\, q}{\mathrm{d}\, \mu}(x)-\inf_{y\in \Omega_0^{x}}\log\frac{\mathrm{d}\, q}{\mathrm{d}\, \mu}(x) +\log(1/\pi(\Omega_0^{x}))\right),\end{align}
and for non-reversible chains, we propose the same except with $\hat{\gamma}$ replaced by $\hat{\gamma}_{\mathrm{ps}}/2$.

%% file: methodology.tex
\section{Methodological considerations for MCMC runs}\label{methodologicalcons}
In this section, we explain how to apply the bounds of Section \ref{sect:concentrationbounds} when we use subsampling or make several parallel runs instead of a single run.
\subsection{Subsampling}
\label{subsampling}
The average $Z:=\left(\sum_{i=t_0+1}^N f(X_i)\right)/(N-t_0)$ is not the only possible way to approximate $\E_{\pi} f$. We may decide to only average in every $m$th step. Assume, without loss of generality, that 
\begin{equation}
N=nm\text{ and }t_0=t_0' m.
\end{equation}
Let $X_1':=X_{m},X_2':=X_{2m}, \ldots , X_n':=X_{n\cdot m}$, and 
\begin{equation}\label{Zpdef}
Z':=\frac{\sum_{i=t_0'+1}^n f(X_i')}{n-t_0'}.
\end{equation}
Then $X_1',\ldots,X_n'$ is a Markov chain, which is reversible if the original chain was reversible, therefore the concentration inequalities from the previous sections are applicable to $Z'$.

The optimal choice of the spacing $m$ was discussed in \cite{geyer1992markov}, which we summarise here as follows. Let $\rho_{i}=\Cov_{X_0\sim \pi}(f(X_0),f(X_i))$, then the asymptotic variance of $Z$ can be written as $\sigma^2=\rho_0+2\sum_{i=1}^{\infty}\rho_i$, and the asymptotic variance of $Z'$ becomes $\sigma^2_m:=\rho_0+2\sum_{i=1}^{\infty}\rho_{i\cdot m}$. If evaluating the function $f$ takes $r$ times more computational time than making an MCMC step, then the optimal $m$ minimises the variance/time ratio $\frac{\sigma^2_m/(N/m)}{(N/m)\cdot r+N}$ (in practice, estimators can be used for $\sigma^2_m$). This can result in considerable speedup, especially when $f$ is much more expensive to evaluate than time it takes to make a step in the Markov chain.

\subsection{Parallel runs}
\label{parallelruns}
In this section, first we will show how we can apply the bounds from the previous sections when averaging over several parallel runs instead of a single run. 
\begin{proposition}[Parallel runs]\label{parallelrunsprop}
Suppose that we have $m$ independent parallel chains $X^{(1)}, \ldots, X^{(m)}$ of length $N$, i.e.\ for $1\le i\le m$, $X^{(i)}=\left(X_{0}^{(i)},\ldots,X_N^{(i)}\right)$ are independent time homogenous Markov chains with initial distribution $X_{0}^{(i)}\sim q$, Polish state space $\Omega$, and stationary distribution $\pi$. We denote by $t_0\ge 0$ the burn-in time. Let $f:\Omega\to \R$, and define the empirical average
\begin{equation}
Z^{(m)}:=\frac{1}{m(N-t_0)}\sum_{i=1}^{m}\sum_{j=t_0+1}^{N}f\left(X_{j}^{(i)}\right).
\end{equation}
Then Theorems \ref{Chebrevthm}, \ref{thmBernsteinrev}, \ref{Chebnonrevthm}, and \ref{NonReversibleBernsteinthmpseudo} apply to $Z^{(m)}$ in the place of $Z$ as well, except that we need to replace $N-t_0$ by $m(N-t_0)$, and $E(t_0)$ by $mE(t_0)$ in each case.
\end{proposition}
\begin{remark}\label{parallelpropremark}
If we choose $t_0$ sufficiently large, then $mE(t_0)$ becomes negligible. In this case, our Proposition says that the bounds for running $m$ chains of length $N$ are equivalent for running a single chain of length $t_0+m(N-t_0)$. Since typically we should choose $N\gg t_0$, this means that there is almost no difference in the bounds when  having a single run or several parallel runs as long as the total number of steps is the same. 
\end{remark}
\begin{proof}
The proof is a simple consequence of the following facts: for independent random variables $Y_1,\ldots,Y_m$, we have $\Var(Y_1+\ldots+Y_m)=\Var(Y_1)+\ldots+\Var(Y_1)$, and
$\E(\exp(\theta(Y_1+\ldots+Y_m)))=\E(\exp(\theta Y_1))\cdot \ldots \cdot \E(\exp(\theta Y_m))$.  The result follows using these together with the variance and moment generating function estimates (which are given in \cite{Martoncoupling}).
\end{proof}

%% file: simulations.tex
\section{Simulations}
\label{sectsimulations}
In the following, we present simulation results to demonstrate the applicability of the introduced error bounds. We are interested in the empirical tail probabilities of estimates, obtained from multiple runs of MCMC simulations. In particular, we will estimate logarithms of tail probabilities of the form
\begin{equation}
\label{eqlogtail}
\log\left(\PP\left(\frac{\sum_{i=t_0+1}^N f(X_i)}{N-t_0}\ge \E_{\pi} f+t\right)\right).
\end{equation}
We simulate $m$ parallel chains started from the same initial distribution $q$, and denote the sequence of states of the $j$th chain ($1 \le j \le m$) by $X_1^{(j)},\ldots,X_N^{(j)}$. 
Then the empirical average obtained by the $j$th chain can be written as
\begin{equation}
\widehat{E}^{(j)}:=\frac{\sum_{i=t_0+1}^{N}f\left(X_i^{(j)}\right)}{N-t_0},
\end{equation}
and denote 
\begin{equation}
\widehat{E}:=\frac{1}{m} \sum_{j=1}^m \widehat{E}^{(j)}.
\end{equation}
We define the \emph{mean-shifted empirical distribution} of these estimates as
\begin{equation}
\widehat{F}(t):=\frac{1}{m}\sum_{j=1}^{m} \II[\widehat{E}^{(j)}-\widehat{E} \le t],
\end{equation}
and let 
\begin{equation}
\label{eqlogtailest}
\widehat{L}(t):= \begin{cases} 
\log\left(\widehat{F}(t)\right) & \text{ for }  t<0 \text{, and } \\
\log\left(1-\widehat{F}(t)\right) & \text{ for } t\ge 0 ,
\end{cases}
\end{equation}
thus $\widehat{L}(t)$ is an estimate of the log tails in \eqref{eqlogtail}.
By the strong law of large numbers, one can see that assuming $X^{(j)}_{t_0+1}\sim \pi$, $\widehat{E}\to \E_{\pi} f$  $\pi$-almost surely as $m\to \infty$. Although $X^{(j)}_{t_0+1}\sim \pi$ does not hold in general, by coupling $X^{(j)}_{t_0+1}$ with a stationary chain, we obtain that
\[\lim\sup_{m\to \infty}\left|\hat{F}(t)-\PP\left(\frac{\sum_{i=t_0+1}^N f(X_i)}{N-t_0}\ge \E_{\pi} f+t\right)\right|\le \dtv(q\mtx{P}^{t_0+1},\pi)\]
$\pi$-almost surely, where $q$ denotes the initial distribution ($X_0\sim q$). Thus if we choose $t_0$ and $m$ sufficiently large, then $\widehat{L}(t)$ estimates well the logarithm of the tail probabilities \eqref{eqlogtail}.

\subsection{Logistic regression}
The space shuttle Challenger exploded during takeoff in 1986, killing all 7 passengers aboard. The weather on the day of the launch was unusually cold, and this was suspected to increase the chance of the failure of the O-ring component. Table 1.1 of \cite{RobertsCasella} shows 23 launch experiments at different temperatures. It is reasonable to try to apply logistic regression to model the dependence of failure on temperature (see Examples 1.13 and 7.11 of \cite{RobertsCasella}). 

Let $Y$ be a random binary response variable, taking values $0$ or $1$ depending on some explanatory variable $x\in \R^p$. Logistic regression models the distribution of $Y$ as
\begin{equation}
\PP(Y=1)=\frac{\exp(\alpha+\left<\beta,x\right>)}{1+\exp(\alpha+\left<\beta,x\right>)},
\end{equation}
where $\alpha\in \R, \beta\in \R^p$ are parameters, and $\left<\cdot,\cdot\right>$ denotes the Euclidean scalar product.

The likelihood of parameters $\alpha, \beta$ given the data $\mathbf{y}$ is
\[L(\alpha,\beta|\mathbf{y})= \prod_{i=1}^{n} \left(\frac{\exp(\alpha+\beta x_i)}{1+\exp(\alpha+\beta x_i)}\right)^{y_i}\left(\frac{1}{1+\exp(\alpha+\beta x_i)}\right)^{1-y_i},\]
where $x_i$ denotes the temperature at the $i$th trial (in Fahrenheit), and $y_i$ denotes the indicator function of the O-ring failure. We choose the prior as $\pi_{\alpha}(\alpha|b)\pi_{\beta}(\beta)=\frac{1}{b}\exp{\alpha} e^{-e^{\alpha}/b}\mathrm{d}\alpha\mathrm{d}\beta$,
which puts an exponential prior on $e^{\alpha}$ and a flat prior on $\beta$, and ensures that the posterior distribution is proper. The parameter $b$ is chosen in a data dependent way as $b=\exp{\hat{\alpha}+\gamma}$, where $\hat{\alpha}$ is the MLE of $\alpha$, and $\gamma\approx 0.577216$ is Euler's constant.

In order to explore this posterior distribution, we use a random walk Metropolis sampler, with normal proposals, having covariance matrix $\left(\begin{matrix}4 & 0\\ 0 & 10^{-3}\end{matrix}\right)$ (these values were obtained after some tuning). We have estimated the spectral gap according to the method of Section \ref{sectionestimatespectralgap}, which yielded $\hat{\gamma}=6.03\cdot 10^{-2}$. We have analyzed the probability of failure at temperatures $30 ^oF$ and $95 ^oF$, denoted by functions $p_{30}(\alpha,\beta)$ and $p_{95}(\alpha,\beta)$. 

The initial distribution was chosen as a uniform distribution on the set 
$\Omega_0=[\hat{\alpha}-0.2,\hat{\alpha}+0.2]\times [\hat{\beta}-0.2,\hat{\beta}+0.2]$.
We have chosen the burn-in period based on inequalities \eqref{Et0Nqbound} and \eqref{Nqbound}. The quantity $\pi(\Omega_0)$ was numerically approximated by computing empirical averages from the second half of an initial run of length $10000$ as $0.13$. 
The quantity $\log\left(\frac{\sup_{x\in \Omega_0}\frac{\mathrm{d}\, \pi}{\mathrm{d}\, q}(x)}{\inf_{x\in \Omega_0}\frac{\mathrm{d}\, \pi}{\mathrm{d}\, q}(x)}\right)$ was found by numerical optimization to be approximately $196.6$. By substituting these approximate values into \eqref{Nqbound},  we obtain that $N_q\le 1.86\cdot 10^{86}$. Using \eqref{Et0Nqbound} with the approximation $\hat{\gamma}$ now yields that $E(t_0)\le 6.5 \cdot 10^{-12}$ if $t_0\ge 2000$,  which is negligible for our purposes. To calculate the empirical log-tails of the estimate, we ran $10^5$ independent runs with $N=10^4$ steps each and a burn-in period of $2000$ steps. The results are shown in Figures \ref{subfig1} and \ref{subfig2}. As we can see, the normal approximation performs poorly in this case by considerably underestimating the error while the Bernstein and Chebyshev bounds work well.

\begin{figure}
\centering
\begin{tabular}{cc}

\addtolength{\subfigcapskip}{0.2cm}
\subfigure[\label{subfig1}\textbf{Logistic regression -- probability of failure at $30 ^oF$.} $\hat{\sigma}^2 = 1.58\cdot 10^{-5}$, $\hat{V}_f = 5.19\cdot 10^{-7}$, $\hat{\gamma}=6.03\cdot 10^{-2}$]{
	\includegraphics[height=4.8cm]{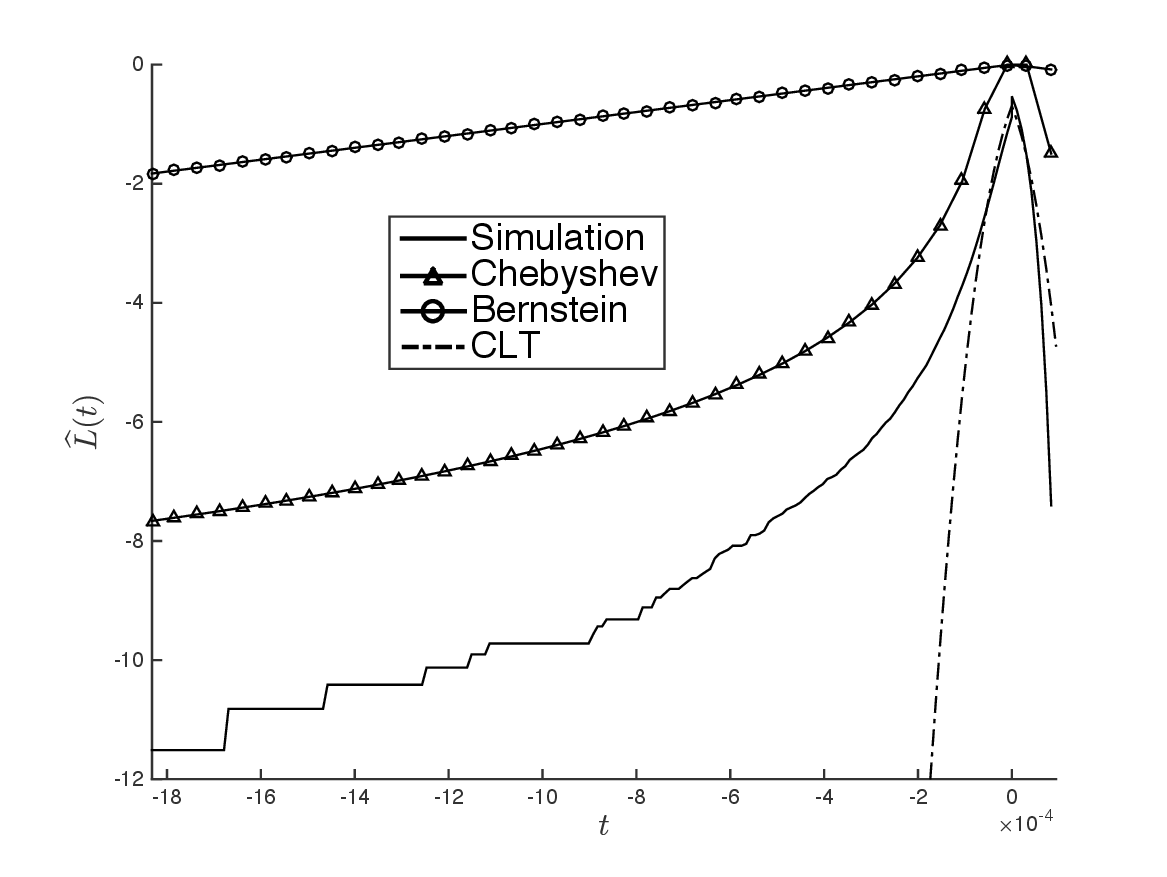}
	}
&
\addtolength{\subfigcapskip}{0.2cm}
\subfigure[\label{subfig2}\textbf{Logistic regression -- probability of failure at $95 ^oF$.} $\hat{\sigma}^2 = 5.29\cdot 10^{-5}$, $\hat{V}_f = 2.28\cdot 10^{-6}$, $\hat{\gamma}=6.03\cdot 10^{-2}$]{
	\includegraphics[height=4.8cm]{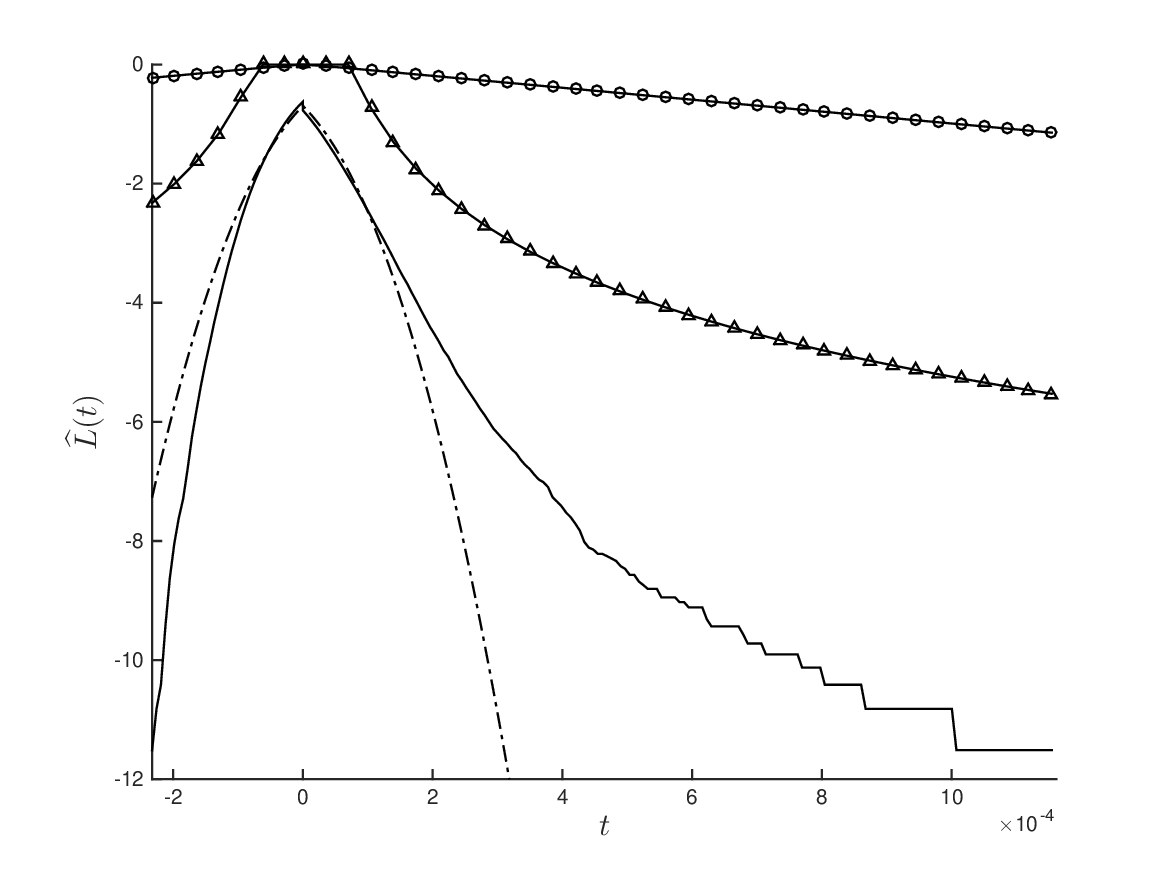}
	}
\\
\addtolength{\subfigcapskip}{0.2cm}
\subfigure[\label{subfig3}\textbf{Competing risk models. } $\hat{\sigma}^2 = 2.96\cdot 10^{-4}$, $\hat{V}_f = 1.90\cdot 10^{-4}$, $\hat{\gamma}=1.07\cdot 10^{-2}$]{
	\includegraphics[height=4.8cm]{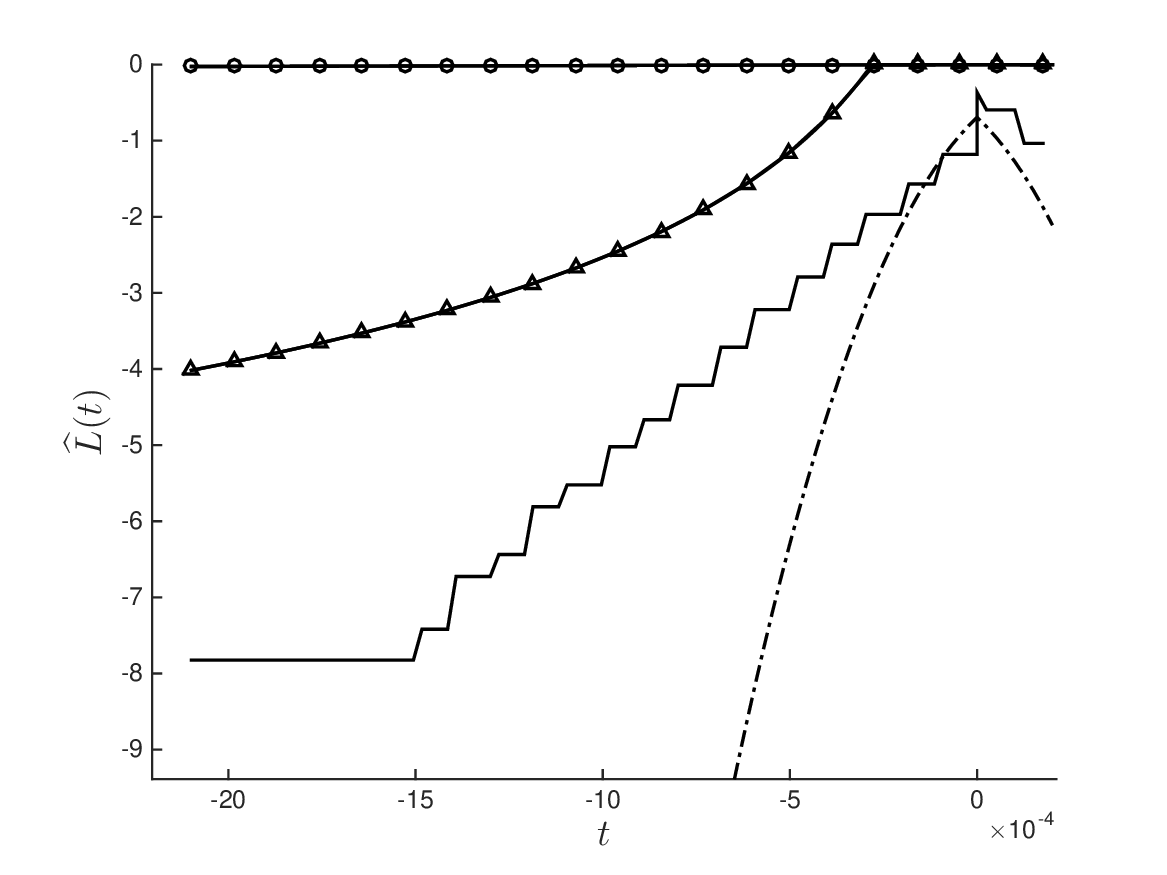}
	}
&
\addtolength{\subfigcapskip}{0.2cm}
\subfigure[\label{subfig4}\textbf{Bayesian clinical trial. } $\hat{\sigma}^2 = 2.55\cdot 10^{-3}$, $\hat{V}_f = 2.76\cdot 10^{-3}$, $\hat{\gamma}_{\mathrm{ps}}=0.817$]{
	\includegraphics[height=4.8cm]{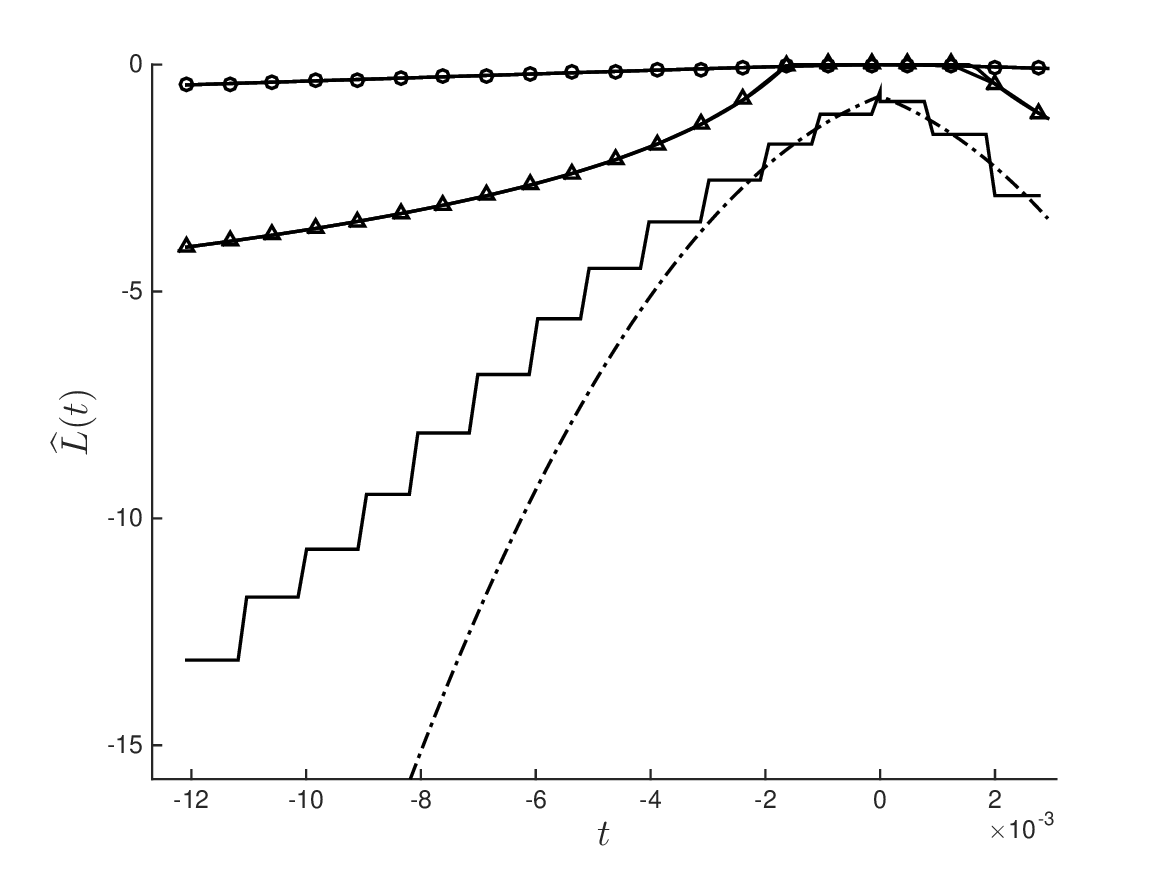}
	}
\end{tabular}
\caption{\textbf{Simulation results for the three examples.} 
The simulation result is plotted according to \eqref{eqlogtailest}. 
We use estimated values of the parameters $\hat{\gamma}$, $\hat{\sigma}^2$ and $\hat{V}_f$ (see Section \ref{estparameters}), 
and plot the Bernstein bound according to \eqref{ReversibleBernsteineq1}. 
as well as the Chebyshev bound according to \eqref{chebempboundrev}. We also show the quantiles of $N(0,\hat{\sigma}^2)$, arising from the CLT (see Section \ref{SecComparison}).}\label{Logisticfig}
\end{figure}
\begin{figure}
\centering
\begin{tabular}{cc}

\addtolength{\subfigcapskip}{0.2cm}
\subfigure{
	\includegraphics[height=4.8cm]{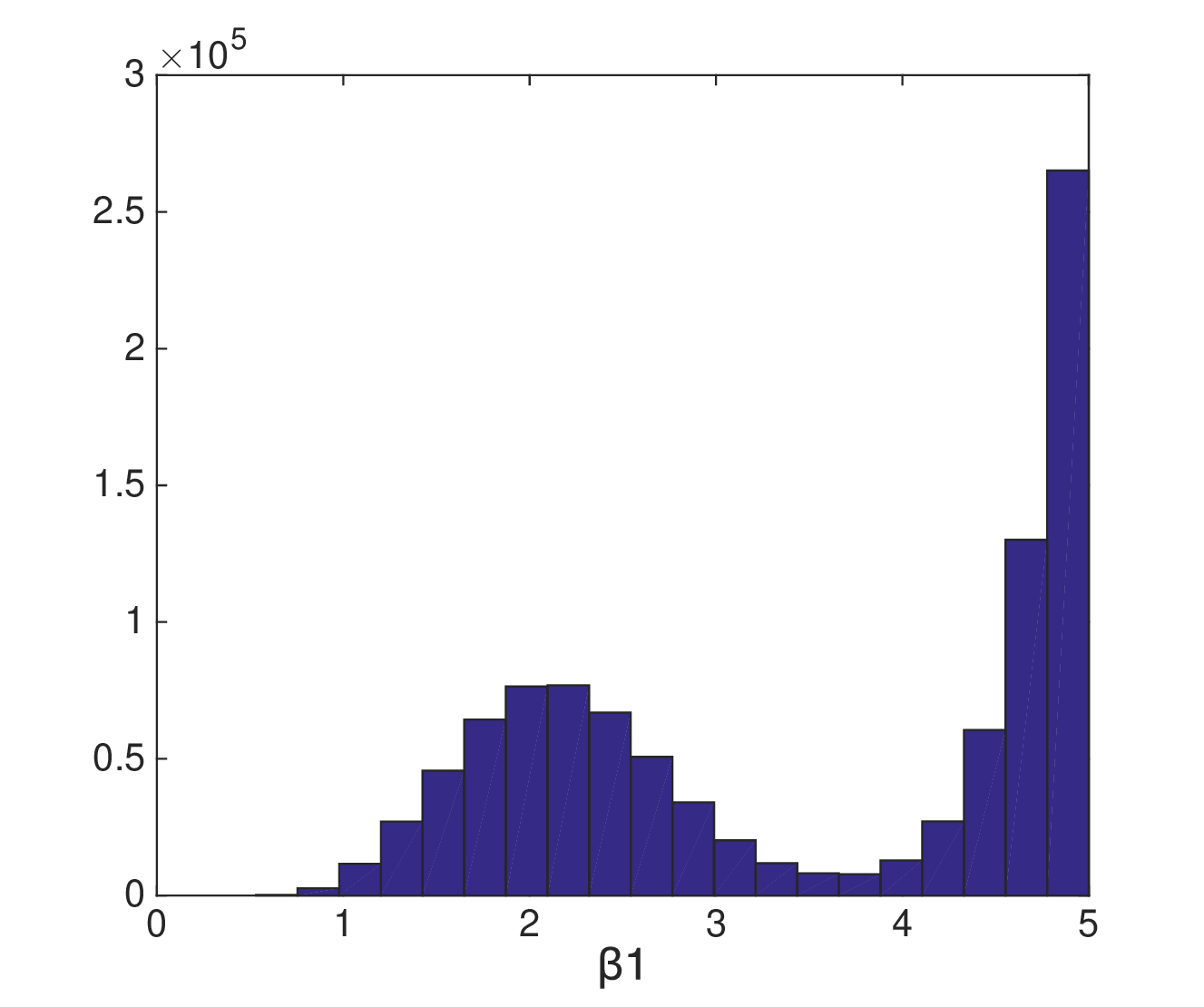}
	}
&
\addtolength{\subfigcapskip}{0.2cm}
\subfigure{
	\includegraphics[height=4.8cm]{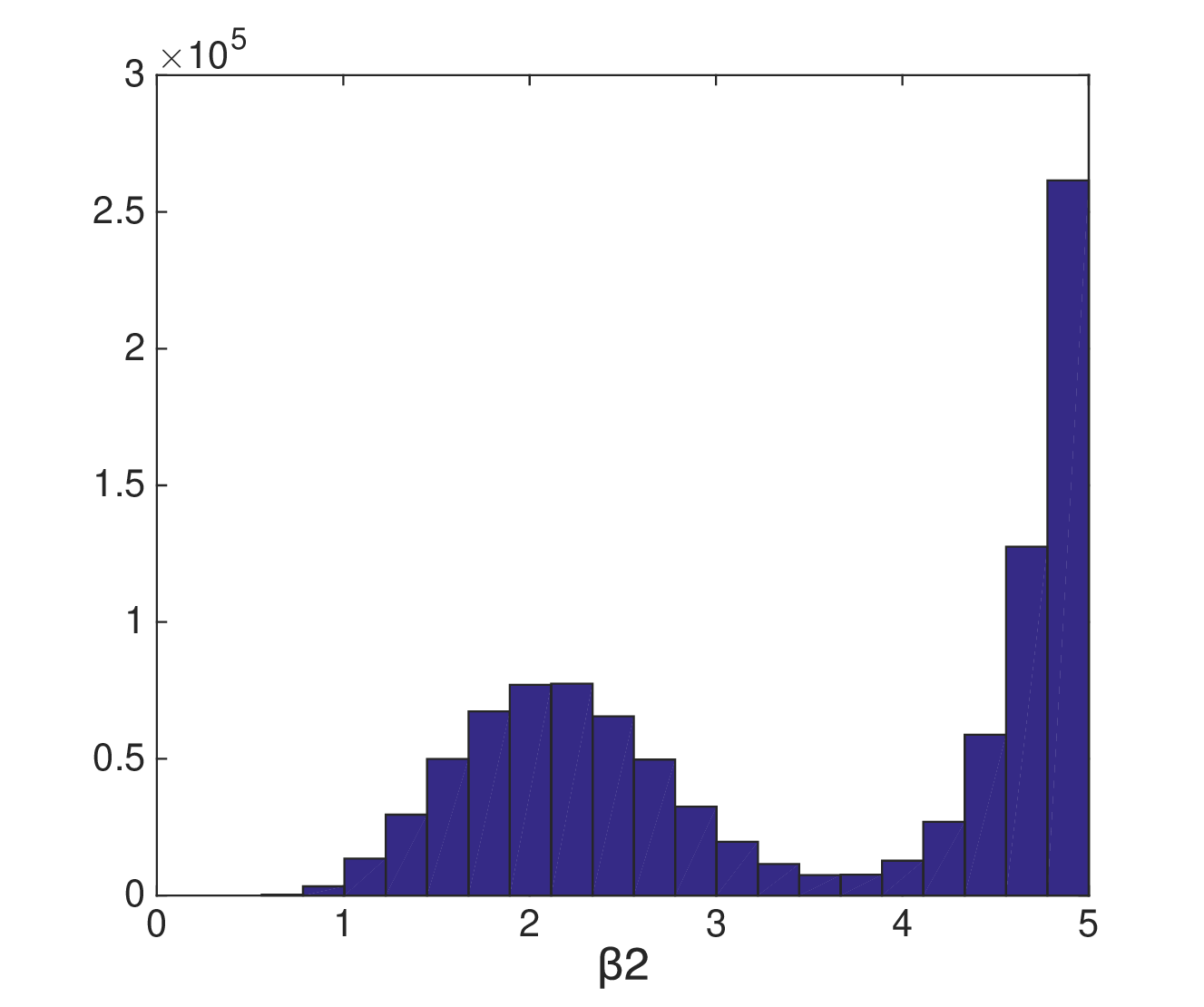}
	}
\\
\end{tabular}
\caption{\textbf{Histograms of marginals of parameters $\beta_1$ and $\beta_2$ for the competing risk model.} }\label{Loglogisticmarginalsfig}
\end{figure}

\subsection{Competing risk models}
Competing risk models attempt to explain the failure time of a system by assuming multiple possible causes of failure with the system failing at the realization of the first one of them. Such models are widely used in survival analysis (see \cite{ibrahim2005bayesian}). In practice we often do not observe the underlying cause of failure. This can lead to non-identifiability, and the resulting posterior distributions can become highly multi-modal due to the label switching problem (see \cite{celeux2000computational}). The simplest MCMC samplers such as Gibbs sampling or random walk Metropolis perform poorly in this context because they rarely make moves between the different modes. However, more advanced samplers such as simulated/parallel tempering are able to overcome these problems (see \cite{marinari1992simulated}, \cite{neal1996sampling}). 

Here we apply our methodology to a parallel tempering MCMC algorithm for a competing risk model implemented based on Section 4.1 of \cite{kozumi2004posterior}. Further details are given in Section \ref{seccompetingdetails} of the Appendix. 
Our model has 4 parameters, ($\beta_1, \sigma_1, \beta_2, \sigma_2$). 
The spectral gap was estimated  by the method of Section \ref{sectionestimatespectralgap} as $\hat{\gamma}=1.07\cdot 10^{-2}$. The initial distribution $q$ was set as uniform distribution on $\Omega_0:=[1,4.6]^{4}$. The probability  $\pi(\Omega_0)$ was estimated from the second half of a $10^5$ long run as $0.014$, and $\log\left(\frac{\sup_{x\in \Omega_0}\frac{\mathrm{d}\, \pi}{\mathrm{d}\, q}(x)}{\inf_{x\in \Omega_0}\frac{\mathrm{d}\, \pi}{\mathrm{d}\, q}(x)}\right)$ was found by numerical optimization to be approximately $331.1$. Based on \eqref{Et0Nqbound} and \eqref{Nqbound} we have that for $t_0:=20000$, $E(t_0)\le 2.10\cdot 10^{-20}$, which is negligible for our purposes. Figure \ref{Loglogisticmarginalsfig} shows the empirical posterior distributions of $\beta_1$ and  $\beta_2$ based on a simulation of length $10^6$ with burn-in of length $t_0$. Clearly, these parameters have multi-modal distributions. 

We chose the function of interest as
$f(\beta_1,\sigma_1,\beta_2,\sigma_2)=\II\left[\E(T|\beta_1,\sigma_1,\beta_2,\sigma_2)>1\right]$, that is, whether the expected survival time of the system is greater than 1.
We ran $m=5000$ parallel chains of length $10^4$, with additional burn-in of length $t_0=20000$, and estimated the empirical mean of $f$ for each run. The log-tails of the empirical averages, as well as the corresponding error bounds are shown in Figure \ref{subfig3}. As we can see, the normal approximation considerably underestimates the error, while the Bernstein and Chebyshev bounds give conservative error estimates.

\subsection{Bayesian methods for clinical trials}
The application of Bayesian analysis for adaptive clinical trials is a relatively new area that holds great promise (see \cite{berry2010bayesian}). Such trials are highly flexible and adaptable to the problems at hand, and they are more powerful than classical parametric methods in many examples. Since they tend to use rather complex models, analytic forms are rarely available and MCMC methods are used for analysis. In this section, we look at a trial proposed for the treatment of diabetes based on pages 211-218 of \cite{berry2010bayesian}. A function of interest $f$, in this case, is the whether the predicted success probability of the treatment exceeds $0.5$. This quantity is very important for the regulatory agency in deciding whether to continue the trial and to declare its success.

The MCMC algorithm updates random variables $\gamma_0, \gamma_1, p, X$ (taking values in $[0,1], [0,1]$, $[0,1]$ and $\{0,1\}^n$, respectively). The sampling follows a Gibbs sampling scheme with systemic scan, which, due to the product nature of the Markov kernel, is non-reversible. Its pseudo spectral gap was estimated using the method of Section \ref{sectionestimatespectralgap} as $\hat{\gamma}_{\mathrm{ps}}=0.817$. 
The initial distribution $q$ was chosen as the uniform distribution on $\left[\frac{1}{168},1-\frac{1}{168}\right]\times \left[\frac{1}{168},1-\frac{1}{168}\right]\times \left[\frac{1}{408},1-\frac{1}{408}\right]\times \{0,1\}^n$.
The burn-in was set as $t_0=600$ based on \eqref{Et0Nqbound} and \eqref{Nqbound}. We provide further details about the sampling scheme, choices of initial distribution and burn-in time in Section \ref{secclinicaldetails} of the Appendix.

To test our methodology on this example, we made $10^6$ parallel runs of length $10^3$ (with an additional burn-in of $600$), and plotted the log-tails of the empirical averages and the error bounds in Figure \ref{subfig4}. Again, the normal approximation underestimates the error while the Chebyshev and Bernstein bounds are conservative. Because of the ethical sensitivity and high cost of clinical trials, we believe that in such situations the additional computational effort needed for using these non-asymptotic bounds is justifiable.

%% file: finalremarks.tex
\section*{Final remarks}
In this paper, we demonstrated the practical usability of concentration inequalities for obtaining non-asymptotically valid error bounds for MCMC empirical averages. We stated Chebyshev and Bernstein-type inequalities for reversible, and non-reversible chains, in a form that is directly applicable in practice. We then proposed estimators for all the parameters arising in the bounds, and gave theoretical explanation for their usage. We have included several examples, which show the advantage of the non-asymptotic approach compared to the asymptotic approach using normal approximation. We have found that especially for indicator functions (i.e.\ when estimating the probability of an event), the normal approximation can drastically underestimate the error, while the Chebyshev and Bernstein bounds are conservative, making these bounds preferable for sensitive problems such as clinical trials and risk assessment. Besides the increased reliability, the spectral gap estimate is also useful for tuning the parameters of the algorithm to improve mixing, and for choosing the initial distribution and the necessary burn-in period.

%% file: appendix.tex
\section{Appendix}\label{sect:appendix}
\subsection{Proof of error bounds for estimators of $V_f$ and $\sigma^2$}
In this section, we will prove our propositions in Section \ref{estparameters}, bounding the error of estimators $\hat{V}_f$ and $\hat{\sigma}^2(k)$. 

\begin{proposition} 
Suppose that $X_1,\ldots,X_{\hat{N}}$ is an uniformly ergodic Markov chain, with stationary distribution $\pi$, and initial distribution $q$. For any $T \ge 0$,
\begin{equation}\label{eqVfhatV}
\PP_q\left(V_f-\hat{V}_f\ge \frac{\tmix}{\hat{N}-\hat{t}_0}+T\right)\le \exp\left( \frac{-(\hat{N}-\hat{t}_0) T^2}{8 C^4 \tmix}\right)+E(\hat{t}_0).
\end{equation}
\end{proposition}

\begin{proof}[Proof of Proposition \ref{propVhatV}]
Changing the function $f$ to $f-\E_{\pi}f$ does not changes $\hat{V}_f$, so we can assume that $\E_{\pi} f=0$, and $|f(x)|\le C$. Now it is easy to show that $\hat{V}_f$ changes at most by $5C^2/(\hat{N}-\hat{t}_0)$ as we change $X_i$. From this, using McDiarmid's bounded differences inequality for Markov chains (Corollary 2.10 of \cite{Martoncoupling}), we can deduce that for any $T\ge 0$,
\[\PP_{\pi}(\hat{V}_{f}-\E_{\pi} \hat{V}_f\ge T)\le \exp\left(-\frac{T^2 (\hat{N}-\hat{t}_0)}{200 \tmix C^4}\right).\]
Moreover, we have
\[V_f- \E_{\pi} \hat{V}_f=\Var_{\pi}\left(\frac{1}{\hat{N}-\hat{t}_0}\sum_{i=\hat{t}_0+1}^{\hat{N}} f(X_i)\right),\]
which, by Theorems 3.1 and 3.2 of \cite{Martoncoupling}, can be further bounded by $2V_f/\gamma/(\hat{N}-\hat{t}_0)$ for reversible chains, and by $4V_f/\gammaps/(\hat{N}-\hat{t}_0)$ for non-reversible chains. Using Proposition \ref{tmixlambdaprop}, these can be further bounded
by $8\tmix/(\hat{N}-\hat{t}_0)$, and the result follows.
\end{proof}

We will use the following lemma for the proof of our propositions about $\hat{\sigma}^2(k)$.
\begin{lemma}
For $t\in \N$, let $\rho_t:=\E_{\pi}[(f(X_1)-\E_{\pi} f)(f(X_{t+1})-\E_{\pi} f)]$.
Then for reversible chains, for $k\ge 2$ even,
\begin{equation}\label{sigmarevgammateq}
-\min\left(\frac{V_f}{2},\frac{2V_f}{\gamma}\cdot (1-\gamma^*)^{k+1}\right)\le \sigma^2-\left(\rho_0+2\sum_{t=1}^{k}\rho_t\right)\le \frac{2V_f}{\gamma}\cdot (1-\min(\gamma,1))^{k+1}.
\end{equation}
For non-reversible chains, we have, for $k\ge 1$,
\begin{equation}\label{sigmanonrevgammateq}
\left|\sigma^2-\left(\rho_0+2\sum_{t=1}^{k}\rho_t\right)\right|\le \frac{4V_f}{\gammaps}\cdot (1-\gammaps)^{(k+1-1/\gammaps)/2}.
\end{equation}
\end{lemma}
\begin{proof}
Without loss of generality, assume that $\E_{\pi} f=0$. Define the operator $\mtx{\pi}$ on $L^2(\pi)$ as $\mtx{\pi}(g)(x):=\E_{\pi}(g)$. 
We have $\sigma^2=\rho_0+2\sum_{t=1}^{\infty}\rho_t$, thus
\begin{align*}
&\sigma^2-\left(\rho_0+2\sum_{t=1}^{k}\rho_t\right)=2\sum_{t=k+1}\rho_t=2\left<f,\left(\sum_{t=k+1}^{\infty}\mtx{P}^t\right)f\right>_{\pi}\\
&=2\left<f, \left(\sum_{t=k+1}^{\infty}(\mtx{P}-\mtx{\pi})^t\right)f\right>_{\pi}
=2\left<f, \left(\mtx{P}-\mtx{\pi}\right)^{k+1}\left(\mtx{I}-(\mtx{P}-\mtx{\pi})\right)^{-1} f\right>_{\pi}.
\end{align*}
For reversible chains, we have $\|\mtx{P}-\mtx{\pi}\|_{2,\pi}\le 1-\gamma^*$, and $\|\left(\mtx{I}-(\mtx{P}-\mtx{\pi})\right)^{-1}\|_{2,\pi}=1/\gamma$, thus
\begin{equation}\label{sigma2gammastareq}
\left|\sigma^2-(\rho_0+2\sum_{t=1}^{k}\rho_t)\right|\le \frac{2V_f}{\gamma}\cdot (1-\gamma^*)^{k+1}.\end{equation}
Moreover, we can express the self-adjoint operator $\left(\mtx{P}-\mtx{\pi}\right)^{k+1}\left(\mtx{I}-(\mtx{P}-\mtx{\pi})\right)$ as a sum of positive and negative parts (we also use the fact that $k+1$ is odd):
\[\left(\mtx{P}-\mtx{\pi}\right)^{k+1}\left(\mtx{I}-(\mtx{P}-\mtx{\pi})\right)^{-1}=\left(\left(\mtx{P}-\mtx{\pi}\right)_+^{k+1}-\left(\mtx{P}-\mtx{\pi}\right)_-^{k+1}\right)\left(\mtx{I}-(\mtx{P}-\mtx{\pi})\right)^{-1}.\]
Now it is easy to see that 
\begin{align*}
&\|\left(\mtx{P}-\mtx{\pi}\right)_+^{k+1}\left(\mtx{I}-(\mtx{P}-\mtx{\pi})\right)^{-1}\|_{2,\pi}\le \min(\gamma,1)^{k+1}/\gamma, \text{ and}\\
&\|\left(\mtx{P}-\mtx{\pi}\right)_-^{k+1}\left(\mtx{I}-(\mtx{P}-\mtx{\pi})\right)^{-1}\|_{2,\pi}\le 1/2,\end{align*}
thus
\begin{align*}&-\min\left(V_f,\frac{2V_f}{\gamma}(1-\gamma^*)^{k+1}\right)\le 2\left<f, \left(\mtx{P}-\mtx{\pi}\right)^{k+1}\left(\mtx{I}-(\mtx{P}-\mtx{\pi})\right)^{-1} f\right>_{\pi}\\
&\le \frac{2V_f}{\gamma}\left(1-\min(\gamma,1)\right)^{k+1}.\end{align*}
Combining this and \eqref{sigma2gammastareq} leads to \eqref{sigmarevgammateq}.

For non-reversible chains, by the proof of Theorem 3.2 of \cite{Martoncoupling}, we can show that
$\|\left(\mtx{I}-(\mtx{P}-\mtx{\pi})\right)^{-1}\|_{2,\pi}\le 2/\gammaps$ and 
$\|\left(\mtx{P}-\mtx{\pi}\right)^{k+1}\|_{2,\pi}\le (1-\gammaps)^{(k+1-1/\gammaps)/2}$, thus \eqref{sigmanonrevgammateq} follows.
\end{proof}

First, we prove the bounds on the bias of $\hat{\sigma}^2(k)$, and then the concentration inequality.
\begin{proof}[Proof of Proposition \ref{sigma2kbias}]
For reversible chains, for $0\le i\le k$, from Chebyshev's inequality (Theorem \ref{Chebrevthm}), we get
\begin{align*}&\left|\E_{\pi}\left(\frac{1}{2}\left(\frac{\sum_{j=t_0+1}^{\hat{N}-k} f(X_j)}{\hat{N}-\hat{t}_0-k}\right)^2+\frac{1}{2}\left(\frac{\sum_{j=t_0+i}^{\hat{N}-k+i} f(X_j)}{\hat{N}-\hat{t}_0-k}\right)^2\right)-\left[(\E_{\pi} f)^2+\frac{\sigma^2}{\hat{N}-\hat{t}_0-k}\right]\right|\\
&\le \frac{4V_f}{\gamma^2}\cdot \frac{1}{(\hat{N}-\hat{t}_0-k)^2},
\end{align*}
and thus it follows that
\[\left|\E_{\pi}(\hat{\rho}_i)-\left(\rho_i-\frac{\sigma^2}{\hat{N}-\hat{t}_0-k}\right)\right|\le \frac{4V_f}{\gamma^2}\cdot \frac{1}{(\hat{N}-\hat{t}_0-k)^2}.\]
Summing up in $i$, and using \eqref{sigmarevgammateq} leads to
\begin{align*}&-K_f-\min\left(V_f,\frac{2V_f}{\gamma}(1-\gamma^*)^{k+1}\right)\le \sigma^2-\left(\hat{\rho}_0+2\sum_{i=1}^{k}\hat{\rho}_i +\frac{\sigma^2(2k+1)}{\hat{N}-\hat{t}_0-k}\right)\\
&\le K_f+\frac{2V_f}{\gamma}\cdot (1-\min(\gamma,1))^{k+1},\end{align*}
where $K_f:=\frac{4V_f}{\gamma^2}\cdot \frac{(2k+1)}{(\hat{N}-\hat{t}_0-k)^2}$. Now putting together the terms involving $\sigma^2$, and dividing by $\frac{\hat{N}-\hat{t}_0-3k-1}{\hat{N}-\hat{t}_0-k}$ leads to \eqref{hatsigmakrevexpectationbound}. The proof of \eqref{hatsigmaknonrevexpectationbound} is similar.
\end{proof}

\begin{proof}[Proof of Proposition \ref{sigma2kerrorbound}]
Firstly, it is easy to show for any $0\le i\le k$, $\hat{\rho}_i$ does not change if we replace the function $f$ by $f-\E_{\pi} f$, thus $\sigma^2(k)$ remains the same under such transformation. Now a simple computation shows that changing the value of $X_j$, for $\hat{t}_0+1\le j\le \hat{N}$, can only change $\hat{\rho}_i$ at most by $8C^2/(\hat{N}-\hat{t}_0-k)$, and thus it can only change the value of 
$\hat{\sigma}^2(k)$ at most by $8(2k+1)C^2/(\hat{N}-\hat{t}_0-3k-1)$. From this (the so called Hamming-Lipschitz property), using McDiarmid's bounded differences inequality for Markov chains (Corollary 2.10 of \cite{Martoncoupling}), we can deduce \eqref{eqconcsigmak}. Finally, \eqref{eqconcsigmakrev} and \eqref{eqconcsigmaknonrev} follow by combining this with the bounds on the bias.
\end{proof}

\subsection{Proof of the mixing time bound for general state spaces}\label{secmixingtimebound}
In this section, we will prove inequality \eqref{eqtmixboundgen}. 

First, using the characterisation $P(x,\cdot)=a(x) q(x,\cdot)+ (1-a(x))\delta_x$, it is easy to see that 
for any integer $k\ge 1$, $P^k(x,\cdot)$ can be written in the form
\[P^k(x,\cdot)=(1-a_k(x))\delta_x + a_{k}(x) \int_{y\in \Omega}c_k(x, \mathrm{d}y) q(x,\cdot),\text{ with }\int_{y\in \Omega}c_k(x, \mathrm{d}y)=1.\]
It is also clear that $(1-a_k(x))\le (1-\inf_{x\in \Omega}a(x))^k$. Therefore using convexity it follows that
\[\sup_{x\in \Omega}\dtv(P^k(x,\cdot),\pi)\le (1-\inf_{x\in \Omega}a(x))^k + \sup_{x\in \Omega}\dtv(P^k q(x,\cdot), \pi).\]
Since $(1-\inf_{x\in \Omega}a(x))^k\le e^{-2}$ for $k\ge \frac{2}{\inf_{x\in \Omega}a(x)}$, for proving \eqref{eqtmixboundgen} it suffices to show that 
$\sup_{x\in \Omega}\dtv(\mtx{P}^k q(x,\cdot), \pi)\le \frac{1}{4}-e^{-2}$ for 
\begin{align*}&k\ge 
\frac{2}{\gamma}+\frac{1}{2 \gamma}\cdot\\
&\sup_{x\in \Omega}\left(\sup_{y\in \Omega_0^{x}}\log\frac{\mathrm{d}\, \pi}{\mathrm{d}\, \mu}(x)-\inf_{y\in \Omega_0^{x}}\log\frac{\mathrm{d}\, \pi}{\mathrm{d}\, \mu}(x) +\sup_{y\in \Omega_0^{x}}\log\frac{\mathrm{d}\, q}{\mathrm{d}\, \mu}(x)-\inf_{y\in \Omega_0^{x}}\log\frac{\mathrm{d}\, q}{\mathrm{d}\, \mu}(x) +\log(1/\pi(\Omega_0^{x}))\right).
\end{align*}
This follows from inequalities \eqref{Et0Nqbound}, \eqref{Nqbound} and the fact that $(1-c)^{d/c}\le \exp(-d)$ for $c\in (0,1), d>0$. The proof for the non-reversible case follows the same lines.

\subsection{Details for the example on competing risk models}\label{seccompetingdetails}
Here we briefly introduce the exact details of the model and the sampler, and show our simulation results.  The actual failure time is $T=\min(X_1,X_2,\ldots,X_k)$, where $X_j$ is the theoretical time for failure for the $j$th cause. We assume that the random variables $X_1, \ldots, X_k$ are independent. $\delta$ denotes the indicator variable whether the failure is observed ($\delta=1$) or right-censored ($\delta=0$). In each observation, we obtain a realization of the two random variables $(T,\delta)$.

The hazard and survival functions of $T$ are 
\[h(t|\theta)=\sum_{j=1}^{k}h_j(t|\theta_j), \quad S(t|\theta)=\prod_{j=1}^k S_j(t|\theta_j).\]
Here $\theta=(\theta_1,\ldots,\theta_k)$,  $\theta_j=(\beta_j,\sigma_j)$ for $1\le j\le k$, and we choose $h_j$ and $S_j$ according to the poly-log-logistic model as
\[h_j(t|\theta_j)=\frac{\beta_j t^{\beta_j-1}}{\sigma_j^{\beta_j}\left\{1+\left(\frac{t}{\sigma_j}\right)^{\beta_j}\right\}}, \quad S_j(t|\theta_j)=\frac{1}{1+\left(\frac{t}{\sigma_j}\right)^{\beta_j}}.\]
Let $(t_1,\delta_1), \ldots, (t_n, \delta_n)$ be $n$ independent observations that might be self-censored, then the likelihood function is written as 
\[L(\theta)=\prod_{i=1}^{n} \left[h(t_i|\theta)^{\delta_i} \cdot S(t_i|\theta)\right].\]

Similarly to Section 4.1 of \cite{kozumi2004posterior}, we have generated a data set from this model by setting $k=2$, $n=50$, $(\beta_1,\sigma_1)=(1,1)$, $(\beta_2,\sigma_2)=(4,4)$. We have chosen a uniform prior on $[0,5]^4$, and ran a parallel tempering MCMC sampler on the posterior. This sampler targets the product distribution $\pi_0\times \ldots\times \pi_{r-1}$. We have chosen $r=10$ tempering distributions, with the final distribution being the posterior distribution $\pi(z)$, and the $i$th one chosen as $\pi_i(z):=\pi(z)^{i/(r-1)}/Z_{i}$ (for some normalizing constant $Z_i$). In each step, when being in location $x=(x_0,x_1,\ldots,x_{r-1})$, we have first made a temperature move, that is, choose $k$ uniformly from $0,1,\ldots, r-2$, and exchanged stages $k$ and $k+1$ (i.e.\ the new location is $(x_0,\ldots, x_{k-1},x_{k+1},x_k,x_{k+2}, \ldots, x_{r-1})$) with probability given by 
\[
\rho(x,k,k+1):=\min\left(1,\frac{\pi_k(x_{k+1})\pi_{k+1}(x_{k})}{\pi_{k}(x_{k})\pi_{k+1}(x_{k+1})} \right),
\]
the so-called ``Metropolis ratio''. After this, when moved according to random walk Metropolis with Gaussian proposals in each stage, and finally made another temperature move.
The resulting sampler is reversible. The initial distribution is chosen as the uniform distribution on $[1,4.6]^4$. The average acceptance rates of the temperature moves between the different stages were $0.0766, 0.5311, 0.6765, 0.7596$, respectively, indicating that the exchanges happened with high probability.

More details and theoretical results on parallel tempering are presented in \cite{Woodardrapidmixing}.

\subsection{Details for the example on clinical trials}\label{secclinicaldetails}
In this section, we include some details on the model and the MCMC sampling scheme for our example about clinical trials, based on Section 5.4 of \cite{berry2010bayesian}. We have implemented the MCMC Algorithm 5.3 on page 214 of \cite{berry2010bayesian} in the same way as in the book\footnote{Available as example5.4.R on \url{http://www.biostat.umn.edu/~brad/software/BCLM_ch5.html}.}

The trial is conducted for $n$ patients in total. The data $Y_i$ is an early indicator variable about the success of the treatment for the $i$th patient after 1 weeks (taking binary values 0 or 1), and $X_i$ is the binary indicator variable corresponding to the primary outcome after 1 month.
We define the vectors $X=(X_1,\ldots, X_n)$, $Y=(Y_1,\ldots, Y_n)$.

There are 3 groups of subjects. For $n_X$ subjects, both $X_i$ and $Y_i$ are observed. For $n_Y$ subjects, we have observed $Y_i$ but not $X_i$. Finally for $n_0$ subjects we have neither $X_i$ nor $Y_i$. The total number of subjects is $n=n_X+n_Y+n_0$.

The trial itself is rather complex, but precisely defined before it is started and no changes are made while it is on-going. If the indicated success rate is sufficiently high after 1 weeks, then the trial is stopped at that point, otherwise we continue for 1 month.

The conditional distributions $X_i | Y_i$ are determined by parameters $\gamma_0$, $\gamma_1$ and $p$ as
\begin{align*}
&\PP(X_i=1|Y_i=0):=\gamma_0,\\
&\PP(X_i=1|Y_i=1):=\gamma_1,\\
&\PP(X_i=1|Y_i\text{ is not observed}):=p.
\end{align*}
We assume that the conditional of $\gamma_0$, $\gamma_1$ and $p$ given the data are Beta distributed. The posterior distribution of the parameters and the unobserved components of the primary outcome data vector $X$ can be expressed as the stationary distribution of the following Markov chain (Algorithm 5.3 of \cite{berry2010bayesian}).
\begin{description}
\item[\textbf{Step 1:}] $\gamma_0$ is drawn from its full conditional distribution, 
\[\gamma_0|X,Y \sim \mathrm{Beta}\left(\alpha_0+\sum_{i=1}^{n_X}\II[X_i=1|Y_i=0],\beta_0+ \sum_{i=1}^{n_X}\II[X_i=0|Y_i=0]\right).\]
\item[\textbf{Step 2:}]$\gamma_1$ is drawn from its full conditional distribution, 
\[\gamma_1|X,Y \sim \mathrm{Beta}\left(\alpha_1+\sum_{i=1}^{n_X}\II[X_i=1|Y_i=1],\beta_1+ \sum_{i=1}^{n_X}\II[X_i=0|Y_i=1]\right).\]
\item[\textbf{Step 3:}] For each subject with $Y_i$ but no $X_i$, an imputed $X_i$ is generated as
\[\PP(X_i=1|Y_i,\gamma_0,\gamma_1)=\gamma_{Y_i}.\]
\item[\textbf{Step 4:}] For each subject with no data, an observed value is simulated as
\[\PP(X_i=1|p)=p.\]
\item[\textbf{Step 5:}] $p$ is simulated from its full conditional distribution
\[p|X\sim  \mathrm{Beta}\left(\alpha+\sum_{i=1}^n \II[X_i=1],\beta+\sum_{i=1}^n \II[X_i=0]\right).\]
\end{description}
The stationary distribution of this Markov chain is the posterior distribution of 
$(\gamma_0,\gamma_1,p,X)$.

Due to the reversibility of each of the steps, the adjoint $P^*$ of this Markov kernel $P$ is defined by repeating the steps in inverse order, starting from Step 5.
Based on this, we have applied the estimation procedure of Section \ref{sectionestimatespectralgap} for the pseudo spectral gap, and obtained the estimate $\hat{\gamma}_{\mathrm{ps}}=0.817$. 

We are going to choose the burn-in time $t_0$, and the initial distribution $q$ based on the bounds \eqref{Et0Nqbound} and \eqref{Nqbound}. The log-likelihood is not directly available in this case, but as we shall see, the changes in the log-likelihood can be bounded nevertheless based on the marginals. The logarithm of the density function of a $\mathrm{Beta}(\alpha,\beta)$ distribution is
\[\log f(x;\alpha,\beta)=\log\left(\frac{\Gamma(\alpha+\beta)}{\Gamma(\alpha)\Gamma(\beta)}\right)+(\alpha-1)\log(x)+ (\beta-1)\log(1-x)\text{ for }0\le x\le 1.\]
This function can be shown to satisfy that for any interval $[a,b]\subset [0,1]$, 
\begin{equation}\label{eqbeta1}\max_{x\in [a,b]}\log f(x;\alpha,\beta)-\min_{x\in [a,b]}\log f(x;\alpha,\beta)\le \max(\alpha-1,\beta-1) \log\left(\frac{b}{a}\right).
\end{equation}
By taking the derivative of $\log f(x;\alpha,\beta)$, one can show that for $\alpha\ge 1$, $\beta\ge 1$, $\alpha+\beta>2$, the maximum is taken at 
$x_{\max}=\frac{\alpha-1}{\alpha+\beta-2}$. One can show that if we suppose in addition that $\alpha$, $\beta$ are integers, then
$f(x_{\max};\alpha,\beta)\le \max(\alpha,\beta)$, implying that for $Z_{\alpha,\beta}\sim \mathrm{Beta}(\alpha,\beta)$,
\begin{equation}\label{eqbeta2}
\PP\left(Z_{\alpha,\beta}\notin \left[\frac{1}{8\max(\alpha,\beta)},1-\frac{1}{8\max(\alpha,\beta)}\right]\right)\le \frac{1}{4}.
\end{equation}

The data and the parameters were chosen according to Tables 5.9 and 5.10 of \cite{berry2010bayesian}. In particular, 
\[n_0=10, n_X=20, n_Y=20, n=50, \alpha_0=\alpha_1=\alpha_p=\beta_0=\beta_1=\beta_p=1.\]
We choose the initial distribution $q$ for $(\gamma_0,\gamma_1,p,X)$ as the uniform distribution on
\[\Omega_0:=\left[\frac{1}{168},1-\frac{1}{168}\right]\times \left[\frac{1}{168},1-\frac{1}{168}\right]\times \left[\frac{1}{408},1-\frac{1}{408}\right]\times \{0,1\}^n,\]
since for this choice, \eqref{eqbeta2} guarantees that $\pi(\Omega_0)\ge \frac{1}{4}$.
Moreover, using \eqref{eqbeta1}, and the marginals in the steps of the Markov chain, one can show that
\[\log\left(\frac{\sup_{x\in \Omega_0}\frac{\mathrm{d}\, \pi}{\mathrm{d}\, q}(x)}{\inf_{x\in \Omega_0}\frac{\mathrm{d}\, \pi}{\mathrm{d}\, q}(x)}\right)\le 20\log(168)+20\log(168)+50\log(408)+50\log(408)\approx 806.1.\]
Therefore it follows from \eqref{Nqbound} that $\log(N_q)<808$. Now \eqref{Et0Nqbound} implies that  based on the estimated value $\hat{\gamma}_{\mathrm{ps}}=0.817$, with the choice $t_0=600$, $E(t_0)<10^{-35}$, which is negligibly small for our purposes.

%% file: main.bbl
\begin{thebibliography}{}

\bibitem[\protect\citeauthoryear{Adamczak}{Adamczak}{2008}]{Adamczaktail}
Adamczak, R. (2008).
\newblock A tail inequality for suprema of unbounded empirical processes with
  applications to {M}arkov chains.
\newblock {\em Electron. J. Probab.\/}~{\em 13}, no. 34, 1000--1034.

\bibitem[\protect\citeauthoryear{Bednorz and {\L}atuszy{\'n}ski}{Bednorz and
  {\L}atuszy{\'n}ski}{2007}]{BednorzLatusz}
Bednorz, W. and K.~{\L}atuszy{\'n}ski (2007).
\newblock A few remarks on ``{F}ixed-width output analysis for {M}arkov chain
  {M}onte {C}arlo'' by {J}ones et al.
\newblock {\em JASA\/}~{\em 102\/}(480), 1485--1486.

\bibitem[\protect\citeauthoryear{Berry, Carlin, Lee, and Muller}{Berry
  et~al.}{2010}]{berry2010bayesian}
Berry, S.~M., B.~P. Carlin, J.~J. Lee, and P.~Muller (2010).
\newblock {\em Bayesian adaptive methods for clinical trials}.
\newblock CRC press.

\bibitem[\protect\citeauthoryear{Brooks, Gelman, Jones, and Meng}{Brooks
  et~al.}{2011}]{HandbookofMCMC}
Brooks, S., A.~Gelman, G.~L. Jones, and X.-L. Meng (Eds.) (2011).
\newblock {\em Handbook of {M}arkov chain {M}onte {C}arlo}.
\newblock Chapman \& Hall/CRC Handbooks of Modern Statistical Methods.

\bibitem[\protect\citeauthoryear{Celeux, Hurn, and Robert}{Celeux
  et~al.}{2000}]{celeux2000computational}
Celeux, G., M.~Hurn, and C.~P. Robert (2000).
\newblock Computational and inferential difficulties with mixture posterior
  distributions.
\newblock {\em JASA\/}~{\em 95\/}(451), 957--970.

\bibitem[\protect\citeauthoryear{Chen, Lov{\'a}sz, and Pak}{Chen
  et~al.}{1999}]{Lovaszlifting}
Chen, F., L.~Lov{\'a}sz, and I.~Pak (1999).
\newblock Lifting {M}arkov chains to speed up mixing.
\newblock In {\em Annual {ACM} {S}ymposium on {T}heory of {C}omputing
  ({A}tlanta, {GA}, 1999)}, pp.\  275--281.

\bibitem[\protect\citeauthoryear{Cowles and Carlin}{Cowles and
  Carlin}{1996}]{cowles1996markov}
Cowles, M.~K. and B.~P. Carlin (1996).
\newblock Markov chain {M}onte {C}arlo convergence diagnostics: a comparative
  review.
\newblock {\em JASA\/}~{\em 91\/}(434), 883--904.

\bibitem[\protect\citeauthoryear{Diaconis, Holmes, and Neal}{Diaconis
  et~al.}{2000}]{diaconisnonrev}
Diaconis, P., S.~Holmes, and R.~M. Neal (2000).
\newblock Analysis of a nonreversible {M}arkov chain sampler.
\newblock {\em Ann. Appl. Probab.\/}~{\em 10\/}(3), 726--752.

\bibitem[\protect\citeauthoryear{Douc, Moulines, Olsson, and van Handel}{Douc
  et~al.}{2011}]{MoulinesMLE}
Douc, R., E.~Moulines, J.~Olsson, and R.~van Handel (2011).
\newblock Consistency of the maximum likelihood estimator for general hidden
  {M}arkov models.
\newblock {\em Ann. Statist.\/}~{\em 39\/}(1), 474--513.

\bibitem[\protect\citeauthoryear{Flegal and Jones}{Flegal and
  Jones}{2010}]{FlegalJones}
Flegal, J.~M. and G.~L. Jones (2010).
\newblock Batch means and spectral variance estimators in {M}arkov chain
  {M}onte {C}arlo.
\newblock {\em Ann. Statist.\/}~{\em 38\/}(2), 1034--1070.

\bibitem[\protect\citeauthoryear{Gelman and Rubin}{Gelman and
  Rubin}{1992}]{gelman1992inference}
Gelman, A. and D.~Rubin (1992).
\newblock Inference from iterative simulation using multiple sequences.
\newblock {\em Statistical science\/}~{\em 7\/}(4), 457--472.

\bibitem[\protect\citeauthoryear{Geyer}{Geyer}{1992a}]{geyer1992practical}
Geyer, C. (1992a).
\newblock Practical {M}arkov chain {M}onte {C}arlo.
\newblock {\em Statistical Science\/}~{\em 7\/}(4), 473--483.

\bibitem[\protect\citeauthoryear{Geyer}{Geyer}{1992b}]{geyer1992markov}
Geyer, C.~J. (1992b).
\newblock {\em Markov chain Monte Carlo maximum likelihood}.
\newblock Defense Technical Information Center.

\bibitem[\protect\citeauthoryear{Gillman}{Gillman}{1998}]{gillman1998chernoff}
Gillman, D. (1998).
\newblock A {C}hernoff bound for random walks on expander graphs.
\newblock {\em SIAM J. Comput.\/}~{\em 27\/}(4), 1203--1220.

\bibitem[\protect\citeauthoryear{Glynn and Ormoneit}{Glynn and
  Ormoneit}{2002}]{GlynnOrmoneit}
Glynn, P.~W. and D.~Ormoneit (2002).
\newblock Hoeffding's inequality for uniformly ergodic {M}arkov chains.
\newblock {\em Statist. Probab. Lett.\/}~{\em 56\/}(2), 143--146.

\bibitem[\protect\citeauthoryear{Gyori and Paulin}{Gyori and
  Paulin}{2015}]{gyori2014hypothesis}
Gyori, B.~M. and D.~Paulin (2015).
\newblock Hypothesis testing for {M}arkov chain {M}onte {C}arlo.
\newblock {\em Statistics and Computing\/}, 1--12.

\bibitem[\protect\citeauthoryear{Hobert, Jones, Presnell, and Rosenthal}{Hobert
  et~al.}{2002}]{Hobert}
Hobert, J.~P., G.~L. Jones, B.~Presnell, and J.~S. Rosenthal (2002).
\newblock On the applicability of regenerative simulation in {M}arkov chain
  {M}onte {C}arlo.
\newblock {\em Biometrika\/}~{\em 89\/}(4), 731--743.

\bibitem[\protect\citeauthoryear{Ibrahim, Chen, and Sinha}{Ibrahim
  et~al.}{2005}]{ibrahim2005bayesian}
Ibrahim, J.~G., M.-H. Chen, and D.~Sinha (2005).
\newblock {\em Bayesian survival analysis}.
\newblock Wil.On.Lib.

\bibitem[\protect\citeauthoryear{Jones, Haran, Caffo, and Neath}{Jones
  et~al.}{2006}]{JonesHaran}
Jones, G.~L., M.~Haran, B.~S. Caffo, and R.~Neath (2006).
\newblock Fixed-width output analysis for {M}arkov chain {M}onte {C}arlo.
\newblock {\em JASA\/}~{\em 101\/}(476), 1537--1547.

\bibitem[\protect\citeauthoryear{Joulin and Ollivier}{Joulin and
  Ollivier}{2010}]{ollivier2010}
Joulin, A. and Y.~Ollivier (2010).
\newblock Curvature, concentration and error estimates for {M}arkov chain
  {M}onte {C}arlo.
\newblock {\em Ann. Probab.\/}~{\em 38\/}(6), 2418--2442.

\bibitem[\protect\citeauthoryear{Kozumi}{Kozumi}{2004}]{kozumi2004posterior}
Kozumi, H. (2004).
\newblock Posterior analysis of latent competing risk models by parallel
  tempering.
\newblock {\em Computational statistics \& data analysis\/}~{\em 46\/}(3),
  441--458.

\bibitem[\protect\citeauthoryear{{\L}atuszy{\'n}ski, Miasojedow, and
  Niemiro}{{\L}atuszy{\'n}ski et~al.}{2013}]{latuszynski2011nonasymptotic}
{\L}atuszy{\'n}ski, K., B.~Miasojedow, and W.~Niemiro (2013).
\newblock Nonasymptotic bounds on the estimation error of {MCMC} algorithms.
\newblock {\em Bernoulli\/}~{\em 19\/}(5A), 2033--2066.

\bibitem[\protect\citeauthoryear{Le{\'o}n and Perron}{Le{\'o}n and
  Perron}{2004}]{leon2004optimal}
Le{\'o}n, C.~A. and F.~Perron (2004).
\newblock Optimal {H}oeffding bounds for discrete reversible {M}arkov chains.
\newblock {\em Ann. Appl. Probab.\/}~{\em 14\/}(2), 958--970.

\bibitem[\protect\citeauthoryear{Levin, Peres, and Wilmer}{Levin
  et~al.}{2009}]{peresbook}
Levin, D.~A., Y.~Peres, and E.~L. Wilmer (2009).
\newblock {\em Markov chains and mixing times}.
\newblock Providence, RI: AMS.
\newblock With a chapter by James G. Propp and David B. Wilson.

\bibitem[\protect\citeauthoryear{Lezaud}{Lezaud}{1998a}]{Lezaud1}
Lezaud, P. (1998a).
\newblock Chernoff-type bound for finite {M}arkov chains.
\newblock {\em Ann.A.P.\/}~{\em 8\/}(3), 849--867.

\bibitem[\protect\citeauthoryear{Lezaud}{Lezaud}{1998b}]{lezaud1998etude}
Lezaud, P. (1998b).
\newblock {\em Etude quantitative des cha{\^\i}nes de {M}arkov par perturbation
  de leur noyau}.
\newblock Th\`{e}se doctorat math\'{e}matiques appliqu\'{e}es de
  l'Universit\'{e} Paul Sabatier de Toulouse, Available at
  \url{http://pom.tls.cena.fr/papers/thesis/these_lezaud.pdf}.

\bibitem[\protect\citeauthoryear{Marinari and Parisi}{Marinari and
  Parisi}{1992}]{marinari1992simulated}
Marinari, E. and G.~Parisi (1992).
\newblock Simulated tempering: a new monte carlo scheme.
\newblock {\em EPL (Europhysics Letters)\/}~{\em 19\/}(6), 451.

\bibitem[\protect\citeauthoryear{Meyn and Tweedie}{Meyn and
  Tweedie}{2009}]{MeynTweedie}
Meyn, S. and R.~L. Tweedie (2009).
\newblock {\em Markov chains and stochastic stability\/} (Second ed.).
\newblock Cambridge: Cambridge University Press.
\newblock With a prologue by Peter W. Glynn.

\bibitem[\protect\citeauthoryear{Miasojedow}{Miasojedow}{2014}]{MiasojedowHoeffding}
Miasojedow, B. (2014).
\newblock Hoeffding's inequalities for geometrically ergodic {M}arkov chains on
  general state space.
\newblock {\em Statist. Probab. Lett.\/}~{\em 87}, 115--120.

\bibitem[\protect\citeauthoryear{Neal}{Neal}{1996}]{neal1996sampling}
Neal, R.~M. (1996).
\newblock Sampling from multimodal distributions using tempered transitions.
\newblock {\em Statistics and computing\/}~{\em 6\/}(4), 353--366.

\bibitem[\protect\citeauthoryear{Paulin}{Paulin}{2014}]{Mixingandconcentration}
Paulin, D. (2014).
\newblock Mixing and concentration by {R}icci curvature.
\newblock {\em arXiv preprint\/}.

\bibitem[\protect\citeauthoryear{Paulin}{Paulin}{2015}]{Martoncoupling}
Paulin, D. (2015).
\newblock Concentration inequalities for markov chains by marton couplings and
  spectral methods.
\newblock {\em Electron. J. Probab.\/}~{\em 20}, no. 79, 1--32.

\bibitem[\protect\citeauthoryear{Robert}{Robert}{1995}]{robertconvergence}
Robert, C.~P. (1995).
\newblock Convergence control methods for {M}arkov chain {M}onte {C}arlo
  algorithms.
\newblock {\em Statist. Sci.\/}~{\em 10\/}(3), 231--253.

\bibitem[\protect\citeauthoryear{Robert and Casella}{Robert and
  Casella}{2004}]{RobertsCasella}
Robert, C.~P. and G.~Casella (2004).
\newblock {\em Monte {C}arlo statistical methods\/} (Second ed.).
\newblock Springer Texts in Statistics. New York: Springer-Verlag.

\bibitem[\protect\citeauthoryear{Roberts and Rosenthal}{Roberts and
  Rosenthal}{2004}]{Robertsgeneral}
Roberts, G.~O. and J.~S. Rosenthal (2004).
\newblock General state space {M}arkov chains and {MCMC} algorithms.
\newblock {\em Probab. Surv.\/}~{\em 1}, 20--71.

\bibitem[\protect\citeauthoryear{Rudolf}{Rudolf}{2012}]{rudolf2011explicit}
Rudolf, D. (2012).
\newblock Explicit error bounds for {M}arkov chain {M}onte {C}arlo.
\newblock {\em Dissertationes Math. (Rozprawy Mat.)\/}~{\em 485}, 1--93.

\bibitem[\protect\citeauthoryear{Woodard, Schmidler, and Huber}{Woodard
  et~al.}{2009}]{Woodardrapidmixing}
Woodard, D.~B., S.~C. Schmidler, and M.~Huber (2009).
\newblock Conditions for rapid mixing of parallel and simulated tempering on
  multimodal distributions.
\newblock {\em Ann.A.P.\/}~{\em 19\/}(2), 617--640.

\end{thebibliography}
